\documentclass[11pt]{article}
\usepackage{geometry}
\DeclareUnicodeCharacter{014D}{\=o}
\DeclareUnicodeCharacter{2032}{\ensuremath{{}^\prime}}
\usepackage{amsmath}
\usepackage{amsthm}
\usepackage{amssymb}
\usepackage{color}
\newtheorem{theorem}{Theorem}
\newtheorem{theorem2}{Theorem}[section]
\newtheorem{assumption}{Assumption}

\newtheorem{corollary}[theorem2]{Corollary}

\newtheorem{lemma}[theorem2]{Lemma}
\newtheorem{proposition}[theorem2]{Proposition}

\newtheorem{definition}[theorem2]{Definition}

\newtheorem{hypothesis}[theorem2]{Hypothesis}%

\newtheorem{remark}[theorem2]{Remark}
\newcommand{\comment}[1]{}
\newcommand{\E}{\mathbb{E}}
\newcommand{\F}{\mathbb{F}}
\newcommand{\G}{\mathbb{G}}

\newcommand{\p}{\mathbb{P}}
\newcommand{\ind}{\mathbf{1}}
\newcommand{\T}{T \wedge \tau}

\DeclareMathOperator*{\esssup}{ess\,sup}
\DeclareMathOperator*{\essinf}{ess\,inf}

\providecommand{\new}[1]{#1}

\newenvironment{newcontent}{\par}{\par}
\AtBeginDocument{\renewcommand{\textcolor}[2]{#2}}

\usepackage{xcolor}
\definecolor{darkgreen}{rgb}{0.0,0.5,0.0}
\providecommand{\addd}[1]{{\color{darkgreen}#1}}

\usepackage{enumitem}
\setlist[enumerate,1]{label=(\arabic*),leftmargin=*}

\title{2BSDE with uncertain horizon and application to stochastic control in erratic environments}
\author{Alberto \textsc{GENNARO}\footnote{alberto.gennaro@berkeley.edu}\; and Thibaut \textsc{MASTROLIA}\footnote{mastrolia@berkeley.edu}}
\date{Department of Industrial Engineering and Operations Research\\ University of California, Berkeley\\ \today}
\begin{document}

\maketitle

\begin{abstract}
We investigate the existence and uniqueness of non-Markovian second-order backward stochastic differential equations with an uncertain terminal horizon and establish comparison principles under the assumption that the driver is Lipschitz continuous in $y,z$ and stochastic Lipschitz in the jump $u$. The terminal time is both random and exogenous, and it may not be adapted to the Brownian filtration, leading to a singular jump in the 2BSDE decomposition. We also provide a connection between this new class of 2BSDE and a fully nonlinear PDE in a Markovian setting. Our theoretical results are applied to non-Markovian stochastic control problems in two settings: (1) when an agent seeks to maximize utility from a payoff received at an uncertain terminal time by controlling both the drift and volatility of a diffusion process; and (2) when the agent contends with volatility uncertainty stemming from an external source, referred to as Nature, and optimizes the drift in a worst-case scenario for the ambiguous volatility. We term this class of problems erratic stochastic control, reflecting the dual uncertainty in both model parameters and the timing of the terminal horizon.
\end{abstract}

\section{Introduction}
\subsection{Backward Stochastic Differential Equations and stochastic control}
Backward Stochastic Differential Equations (BSDE for short) are a class of stochastic equations characterized by a terminal condition $\xi$ and a generator function $F$, which depend on time $t$, randomness $\omega$ and space variables $y,z$. On a probability space $(\Omega,\mathcal F,\mathbb P)$ equipped with a Brownian motion $W$. A solution to a BSDE parametrized by $(\xi,F)$ is a pair of stochastic processes $(Y,Z)$, adapted to the filtration of the Brownian motion satisfying the following backward equation\footnote{All equations in this introduction are presented informally and without full mathematical rigor. Their purpose is solely to motivate the study; rigorous treatment follows in the main body of this work.}
\[Y_t = \xi -\int_t^T F(s,\omega,Y_s,Z_s)ds-\int_t^T Z_s dW_s,\; t\leq T,\; \mathbb P\text{-a.s.}\]
A very special case arises when $\xi=g(X_T)$ where $X$ is a solution to a stochastic differential equation and $F(s,\omega,Y_s,Z_s)=F(s,X_s(\omega),Y_s(\omega),Z_s(\omega))$, the corresponding BSDE is set in a Markovian framework. The well-posedness of a Markovian BSDE can thus be reduced to the existence of a solution in a PDE where $Y$ is the solution and $Z$ is its gradient, see, for example, \cite{ma2002representation}. This type of equation was first introduced by Bismut in \cite{bismut1973conjugate}, proving the existence and uniqueness of the solution when the generator $F$ is assumed to be linear in $Y$ and $Z$.  It has then been generalized in the seminal paper \cite{pardoux1990adapted} under Lipschitz conditions in a Markovian setting, then extended to a non-Markovian framework with applications in finance (pricing and hedging of derivatives) in the celebrated article\cite{el1997backward} or to more general assumptions to the generator $F$, like stochastic Lipschitz or quadratic growth \cite{bahlali2001backward,briand2008bsdes,briand2008quadratic} or extended to jump processes \cite{buckdahn1994bsde,rong1997solutions,quenez2013bsdes,papapantoleon2018existence}. \textcolor{blue}{Since these foundational studies, BSDEs have been extensively investigated, together with their links with stochastic control and applications in diverse fields of applied mathematics including but not limited to utility maximization, see for example \cite{el1997backward,rouge2000pricing,hu2005utility,ceci2014bsdes,morlais2009quadratic} and games theory.} We refer more generally to \cite{yong1999backward,touzi2012optimal} for more details of the BSDE theory.\\

In this work, we focus on a stochastic control problem, in a non-Markovian setting in which a diffusion process $X$ is controlled both in drift and volatility:
\[dX_t=\mu(t, X, \alpha_t)dt+\sigma(t,X,\beta_t) dW_t,\; \nu=(\alpha,\beta).\]
Echoing the analysis presented in \cite[Chapter 9]{zhang2017stochastic} and \cite[Chapter 4]{yong1999backward}, different types of control strategies can be considered in stochastic control problems. When the control process is open-loop - meaning it depends on time $t$ and the underlying Brownian motion $W$ - the problem is naturally framed in a \emph{strong formulation}. This setting leads to the study of Forward-Backward Stochastic Differential Equations (FBSDEs for short), and the optimal control can be characterized using the stochastic maximum principle, originally formulated by Pontryagin in the 1950s; see, for instance, \cite{yong1999backward,oksendal2010maximum}. Alternatively, in the case of \emph{closed-loop} (or feedback) control, where the control depends on $t$ and the controlled state process $X$, the problem is typically treated in a \emph{weak formulation}. This framework often arises in contexts such as Stackelberg games or contract theory \cite[Section 4]{cvitanic2018dynamic}. In this setting, Girsanov’s theorem is employed to shift to a family of equivalent probability measures $\mathbb{P}^{\alpha}$, under which a new Brownian motion is defined. This change of measure facilitates the analysis of the control problem via BSDE.

\subsection{Second order BSDE and volatility-controlled or ambiguous optimization}

Controlling the volatility becomes more challenging in the weak formulation setting, as it involves changing the probability measure associated with a diffusion process and its volatility. This requires constructing a family of non-equivalent probability measures, denoted by $\mathbb P^{\alpha,\beta}$ through a weak formulation and weak solution of the controlled SDE considered. The optimal control problem then consists of maximizing an objective function over a set of probability measures parameterized by the control variables $\alpha$ and $\beta$. Under each such measure, the objective function is either a supermartingale or a submartingale, depending on the nature of the problem. The optimal measure corresponds to the one under which a martingale optimality principle is satisfied. As an extension of the BSDE introduced earlier, we then expect the problem to be reduced to a second-order BSDE (2BSDE for short), with terminal condition $\xi$ and generator $F$ given by 
\[Y_t=\xi-\int_t^T \hat F_s(Y_s,Z_s)ds -\int_t^T \hat{\sigma}_s^{1/2}Z_s dX_s+K_T-K_t,\; t\leq T,\; \mathbb P\text{-a.s.}, \forall \mathbb P\in \mathcal P,\]
where $\mathcal P$ is a set of non-equivalent probability measures and $\hat F$ is the generator of the 2BSDE determined through the universal representation of the quadratic variation $\hat{\sigma}_t$ of $X$ as soon as the family $\mathcal P$ satisfies regularity properties, see \cite{karandikar1995pathwise, nutz2015robust}. A solution to this 2BSDE consists of a triple $(Y,Z,K)$ of stochastic processes where the $K$ process represents the second-order representation, is non-decreasing, and satisfies the minimality condition
\[0=\essinf_
{\mathbb P}\mathbb E^\mathbb P[K_T-K_t|\mathcal F_t].\]
The existence and uniqueness of solutions to 2BSDEs were first studied in \cite{soner2012wellposedness,cheridito2007second}, and later extended to more general frameworks in \cite{kazi2015second,possamai2015weak,lin2016new,kim2024representation}, under the assumption that both the terminal condition and the generator of the 2BSDE satisfy strong continuity conditions with respect to $\omega$.
An extension to random horizons, where the terminal time is adapted to the diffusion filtration, was developed in \cite{lin2020second}. The connections between 2BSDEs and stochastic games or control problems involving controlled volatility or model uncertainty (ambiguity) have been explored in, for example, \cite{matoussi2014second,matoussi2015robust,popier2019second,mezdoud2021alpha}.
Fundamental advances were made in \cite{possamai2018stochastic}, where the regularity assumptions on both the terminal condition and the generator were relaxed using a more general class of probability measures. This was further extended in \cite{lin2020second,hun2024wellposedness,denis2024second}. These relaxations have been crucial in solving Stackelberg games and stochastic optimization problems involving controlled volatility or ambiguity, see for example, \cite{cvitanic2018dynamic,mastrolia2018moral,hernandez2019contract,lin2022random,mastrolia2025agency}.

\subsection{BSDE with erratic horizon}
In this work, we are interested in stochastic control problems where the horizon is not only random but also unpredictable and erratic, in the sense that it is exogenous and not adapted to the filtration of the controlled process $X$. To address this difficulty, we rely on the theory of filtration enlargement, which allows us to incorporate additional information into the available filtration by accounting for the extra information provided by the termination time. This type of problem is particularly relevant in the context of credit risk, where a firm’s default is modeled as a random time that is not adapted to the filtration generated by observable market variables. See, for example, \cite{kabanov2006enlargement,acciaio2016arbitrage}. We refer to 
\cite{jeulin2006grossissements, bielecki2013credit,aksamit2017enlargement} for a non-exhaustive list of monographs on enlargement of filtration. In a non-Markovian setting, optimization problems involving an erratic time $\tau$ have been investigated in \cite{kharroubi2013mean,jeanblanc2014note,jeanblanc2015utility}. It is proven that one can reduce the optimization to solve a BSDE with an erratic horizon with a single jump of the form:
\begin{equation}\label{BSDEintro}
Y_t = \xi - \int_{t \wedge \tau}^{T \wedge \tau} Z_s \cdot dW_s 
      - \int_{t \wedge \tau}^{T \wedge \tau} U_s dH_s 
      - \int_{t \wedge \tau}^{T \wedge \tau} F(s, Y_s, Z_s, U_s) ds, 
      \quad t \in [0, T],\; \mathbb P\text{-a.s.},
\end{equation}
where $H_s=\mathbf 1_{\tau\leq s}$ indicates if a random default occurs or not. Under different assumptions regarding the default time $\tau$, the generator of the associated BSDE may become degenerate. This type of degeneracy has been extensively studied in the literature, particularly in the context of insurance. See, for example, \cite{kruse2016minimal,huang2020kind,aurand2023epstein,di2020weak,jiao2018information}. In particular, \cite[Theorem 4.1]{kharroubi2013mean} or \cite[Proposition 4.4]{jeanblanc2014note} prove that the solution $(Y,Z,U)$ of this BSDE can be reduced to solve a Brownian driven BSDE (possibly with a degenerated generator) given by
\[
Y^b_t = \xi^b 
- \int_t^T F^b(s, Y^b_s, Z^b_s, \xi^a_s - Y^b_s) \, ds 
- \int_t^T Z^b_s \cdot dW_s, 
\quad t \in [0, T],
\]
and the solution to \eqref{BSDEintro} is given by 
\begin{align*}
Y_t = Y^b_t \mathbf{1}_{\{t < \tau\}} + \xi^a_\tau  \mathbf{1}_{\{t \geq \tau\}},\, Z_t = Z^b_t  \mathbf{1}_{\{t \leq \tau\}},\, U_t = (\xi^a_t - Y^b_t)  \mathbf{1}_{\{t \leq \tau\}},
\end{align*}
where $\xi^b,\xi^a,F^b$ are representations of $\xi$ and $F$ with respect to the event $\{\tau\leq T\}$. 
This result is however only related to optimal control problems for either exponential portfolio utility maximization similar to \cite{hu2005utility} or drift-controlled diffusion. The goal of this study is to extend it to volatility-controlled or ambiguity models.

\subsection{Contribution: 2BSDE with erratic horizon and erratic stochastic control}

\paragraph{Motivation with crypto trading and cyber risk.}
We define \textbf{erratic stochastic control} as a class of control problems in which the decision-making process must respond to highly irregular, unpredictable stochastic environments and potentially abrupt termination events. The cryptocurrency market illustrates erratic stochastic behaviors. In this domain, asset prices are influenced by extremely volatile and unpredictable factors, including random market noise, sudden regulatory interventions, and unforeseen news events. These dynamics result in behavior that is difficult to capture or forecast using traditional modeling techniques. Panic-driven behaviors have partially led to the cascading collapse of Terra (LUNA) and its associated stablecoin UST in May 2022, which triggered a rapid market-wide reaction \cite{briola2023anatomy}. Cryptocurrency markets frequently undergo sudden regime shifts— such as speculative bull runs, dramatic crashes, or fraudulent events as rug pulls. Market participants often operate under conditions of Knightian uncertainty, where the underlying probability distributions are unknown, ill-defined, or even unknowable, see \cite{almeida2022uncertainty}. Tracking 146 Proof-of-Work-based cryptocurrencies that started trading prior to
the end of 2014 until December 2018, \cite{grobys2020predicting} shows that about 60\% of those
cryptocurrencies were eventually in default. In cyber-risk, Knightian uncertainty plays a crucial role, as cyber-attacks introduce greater volatility and unpredictable conditions into the system. For example, see \cite{cheraghali2024impact,mastrolia2025agency}.

\paragraph{Contribution and structure.} In this work, we address the problem of stochastic control in erratic environments by incorporating both control and ambiguity in the volatility of a controlled diffusion, as well as high uncertainty in the termination of the optimization problem. We begin by proposing mathematical models in Section~\ref{sec:math}, where we extend the so-called Density Hypothesis in the theory of enlargement of filtrations to a $\mathcal{P}$-density assumption that holds for the default time under a set of probability measures. We then investigate the existence, uniqueness, and comparison results for solutions in Section~\ref{sec:2BSDE} to a second-order backward stochastic differential equation (2BSDE) with an erratic horizon of the form:
\begin{align*}
Y_t &= \xi - \int_{t \wedge \tau}^{T \wedge \tau} F_s^{\mathbb{P}}(Y_s, \hat{\sigma}_s^{1/2} Z_s, U_s) , ds - \int_{t \wedge \tau}^{T \wedge \tau} Z_s \cdot dX_s^{c,\mathbb{P}} + \int_{t \wedge \tau}^{T \wedge \tau} dK_s - \int_{t \wedge \tau}^{T \wedge \tau} U_s , dH_s - \int_{t \wedge \tau}^{T \wedge \tau} dM_s^{\mathbb{P}},\\
&\quad \text{for all } 0 \leq t \leq T,\; \mathbb{P}\text{-a.s., for every } \mathbb{P} \in \mathcal{P}_0.
\end{align*}
This section extends existing results on BSDEs with random horizons and 2BSDEs with finite or random termination to the setting of 2BSDEs with an erratic horizon. Our formulation also allows for the aggregation of the $U$ process, in contrast to existing theories on 2BSDEs with jumps; see Remark~\ref{sec:aggregation}. We conclude this theoretical section with a connection to fully nonlinear PDE in Section~\ref{section:PDE}. Finally, we investigate erratic stochastic control in Section~\ref{sec:erraticcontrol} and explore its connection to 2BSDEs with erratic horizons. We first present a comprehensive and general structure for the erratic control problem in Section~\ref{sec:controlmodel}, unifying the volatility control and ambiguity models within our framework. We then define the 2BSDE associated with both problems in Section~\ref{sec:control2bsde}. The volatility control problem is studied in Section~\ref{sec:control}, followed by an analysis of the Knightian uncertainty problem in Section~\ref{Section:ambiguity}.

\begin{newcontent}
\paragraph{Methodological contributions.}\label{par:methodological-contribution}
 The principal technical
challenge is that the jump caused by the default $\tau$ at the random horizon
$T\wedge\tau$ interacts with the second-order minimality term
$K^{\mathbb P}$ across a non-dominated family of measures
$\mathcal P_0$. The standard before/after-default decomposition of
\cite{kharroubi2013mean,jeanblanc2014note} fixes a single reference
measure and uses the Jacod density of $\tau$ under it, while the
2BSDE machinery requires a kernel that is
uniquely and simultaneously defined for every $\mathbb P\in\mathcal P_0$, including
those that are mutually singular.  Our methodological contribution
is the introduction of the
$\mathcal P_0$-density Hypothesis~\ref{densityhyp}, which forces the conditional density $\gamma$ of
$\tau$ to be the same kernel under every measure in $\mathcal P_0$ and
is what allows us to (i) aggregate the $\mathbb F$-intensity $\lambda$
across $\mathcal P_0$, see Remark \ref{rem:agglambda} (ii) aggregate the jump-size process $U$ in the
primal 2BSDE \eqref{2bsde:eq} see Remark~\ref{sec:aggregation},
(iii) carry out the auxiliary 2BSDE construction
(Lemma~\ref{thm:auxiliaryBSDE}--Theorem~\ref{thm:decompbsde}) that
reduces the random-horizon 2BSDE to a Brownian-driven 2BSDE on
$[0,T]$ with a frozen post-default boundary; (iv) extension to CARA utility non-Markovian maximization problem to erratic horizon with volatility ambiguity, see Appendix \ref{appendix:cara}.  The $\mathcal
P_0$-density assumption is strictly stronger than Jacod's classical
density hypothesis under a fixed measure to ensure the well-posedness of the 2BSDE, and we exhibit two structurally
different families of examples (equivalent measures via Girsanov;
$\mathbb P$-invariant law of $\tau$) under which it is satisfied.
The connection to fully nonlinear PDE in Section~\ref{section:PDE}
shows that, despite the jump and the non-Markovian structure, the
random horizon does not enter the PDE itself: the dependence on
$\tau$ is absorbed into a piecewise Feynman--Kac representation, with
the singular value of $K$ at $\tau$ pushed to a boundary condition.
This is what we mean by ``erratic stochastic control'' as a
methodologically distinct framework, as opposed to a relabeling of
robust stochastic control: the family $\mathcal P_0$ is non-dominated
\emph{and} the horizon is exogenous.
\end{newcontent}

\section{Mathematical framework and Definitions}
\label{sec:math}

\subsection{Probabilistic framework}
Let $d$ be a positive integer and set $\Omega:=\mathbb D([0,T];\mathbb R^d)$ defined by the set of c\`adl\`ag paths   
on $\mathbb R^{d}$ starting at $0$, equipped with the Skorohod topology. We denote by $X$ the canonical process on this space such that
$X_t(\omega)=\omega_t$ for all $\omega\in \Omega$. We denote by $\mathbb{F} = (\mathcal{F}_t)_{0 \leq t \leq T}$ the filtration generated by the canonical process \new{$X$}, and by  $\mathbb{F}^+ = (\mathcal{F}^+_t)_{0 \leq t \leq T}$ the right limit of $\mathbb{F}$, defined by
\[
\mathcal{F}^+_t := \bigcap_{s > t} \mathcal{F}_s \quad \text{for all } t \in [0, T), \quad \text{and} \quad \mathcal{F}^+_T := \mathcal{F}_T.
\] We also define 
\[
\mathcal{F}^\mathcal{P}_t := \bigcap_{\mathbb P \in \mathcal{P}} \mathcal{F}^{\mathbb P}_t, \quad 
\mathcal{F}^{\mathcal{P}+}_t := \mathcal{F}^\mathcal{P}_{t+},\; \mathbb F^{\mathcal P_+}:=(\mathcal{F}^{\mathcal{P}+}_t)_{t\leq T},
\]
where $\mathcal P$ is a set of probability measures on $\Omega$.

\begin{definition}[Semi-martingale measures]
 We define $\mathcal P_{t}^W$ as the set of all semi-martingale measures
$\mathbb P$ on $(\Omega,\mathcal F)$, such that $\mathbb P$-a.s.:
\begin{enumerate}
    \item $(X_s)_{s\in [t,T]}$ is a $(\mathbb P; \mathbb F)-$semi-martingale admitting the canonical decomposition
    \[X_s=X_t+\int_t^s \mu^\mathbb P_r dr+X_s^{c,\mathbb P},\] where $\mu^\mathbb P$ is an $\mathbb F-$predictable $\mathbb R^d-$valued process and $X^{c,\mathbb P}$ denotes the continuous local martingale part of $X$ under $\mathbb P$;
\item the quadratic variation process $(\langle X\rangle_s)_{s\in [t,T]}$ is absolutely continuous in time $s$ with respect to the Lebesgue measure and the process $\hat{\sigma}$ defined by
\[\hat{\sigma}_t:= \lim\sup_{\varepsilon\to 0}\frac{\langle X\rangle_t-\langle X\rangle_{t-\varepsilon}}{\varepsilon}, \] takes values in $\mathcal S_d^{\geq 0}$, $\mathbb P\text{-a.s.}$, where $\mathcal S_d^{\geq 0}$ denotes the space of symmetric and positive definite matrix in dimension $d$.
\end{enumerate}
\end{definition}
In this study, we consider a family $\big(\mathcal P(t,\omega)\big)_{(t,\omega) \in [0,T] \times \Omega}$ of probability measures on $(\Omega, \mathcal{F}_T)$, where $\mathcal P(t,\omega) \subset \mathcal P^W_t$ for all $(t,\omega) \in [0,T] \times \Omega$ satisfying (iii)-(iv)-(v) in \cite[Assumption 1.1]{possamai2018stochastic} or \cite{nutz2013constructing} that we recall below
\begin{itemize}
    \item For every $(t,\omega) \in [0,T] \times \Omega$, one has $\mathcal{P}(t,\omega) = \mathcal{P}(t,\omega_{\cdot \wedge t})$ and $\mathbb P(\Omega_t^\omega) = 1$ whenever $\mathbb P \in \mathcal{P}(t,\omega)$, where
    \[\Omega_\tau^\omega := \{\omega' \in \Omega : \omega'(s) = \omega(s),\ 0 \leq s \leq \tau(\omega)\}.\]The graph $[\![\mathcal{P}]\!]$ of $\mathcal{P}$, defined by
    \[
        [\![\mathcal{P}]\!] := \{(t,\omega,\mathbb P) : \mathbb P \in \mathcal{P}(t,\omega)\},
    \]
    is analytic in $[0,T] \times \Omega \times \mathcal{M}_1$ in the sense of \cite[Section~3]{nutz2013constructing}.

    \item for every $(t,\omega) \in [0,T] \times \Omega$ and every $\mathbb P \in \mathcal{P}(t,\omega)$ together with an $\mathbb{F}$-stopping time $\tau$ taking values in $[t,T]$, there is a family of regular conditional probability distribution in the sense of \cite{stroock2007multidimensional}, denoted by \ $(\mathbb P_\omega)_{\omega \in \Omega}$ such that
        $\mathbb P_\omega \in \mathcal{P}(\tau(\omega),\, \omega), \quad \text{for } \mathbb P\text{-a.e. } \omega \in \Omega.$

    \item for every $(t,\omega) \in [0,T] \times \Omega$ and $\mathbb P \in \mathcal{P}(t,\omega)$ together with an $\mathbb{F}$-stopping time $\tau$ taking values in $[t,T]$, let $(\mathbb Q_\omega)_{\omega \in \Omega}$ be a family of probability measures such that $\mathbb Q_\omega \in \mathcal{P}(\tau(\omega),\,\omega)$ for all $\omega \in \Omega$ and $\omega \longmapsto \mathbb Q_\omega$ is $\mathcal{F}_\tau$-measurable, then the concatenated probability measure $\mathbb P \otimes_\tau \mathbb Q_\cdot \in \mathcal{P}(t,\omega).$
\end{itemize}
Denote also
\[
\mathcal P_t := \bigcup_{\omega \in \Omega} \mathcal P(t,\omega),
\qquad
\mathcal P_0 := \mathcal P_{t=0}.
\]
\begin{remark}
    $\mathcal P(t,\omega)$ represents a \emph{local} family used in the
dynamic-programming construction of the solution to a 2BSDE and is taken as a general set of probabilities in $\mathcal P_t^W$ while 
$\mathcal P_0=\bigcup_\omega\mathcal P(0,\omega)$ is the \emph{global} family on
which the well-posedness statement and Hypothesis~\ref{densityhyp} below are imposed.
\end{remark}

\subsection{Enlargement of filtration}
The default time is represented by a random variable $\tau$ taking values in $\mathbb R^+$. We define the default process by $H_t := \ind_{\tau \leq t}$. Note that this process is not necessarily $\mathbb F-$ measurable, since the default time is assumed to be potentially exogenous to the system and therefore independent of $X$. We thus enlarged the available information and so the filtration $\mathbb F$ taking into account the information generated by the default time occurrence. 
\begin{definition}
   Let $\mathcal H_u:=\sigma(H_s, s \in [0, u])$ be the $\sigma-$algebra generated by $H$ until time $u\geq 0$. Given a filtrated space $(\Omega, \mathcal{F}_T, \F, \mathbb{P})$, the enlarged filtration
    \[
        \G = (\mathcal{G}_t)_{t \in [0,T]}, \qquad \mathcal{G}_t = \underset{\epsilon > 0}{\bigcap} \{ \mathcal{F}_{t+\epsilon} \vee \mathcal H_{t +\epsilon} \},
    \]
    is the smallest enlargement of $\F$ such that $\tau$ is a $\G$-stopping time.
\end{definition}

\begin{remark}
$H$ is not $\mathbb F-$measurable but it is $\G$-measurable stochastic process. 
\end{remark}
The goal is to ensure that the inaccessible default time $\tau$ enables us to enlarge the filtration to $\mathbb G$ and transfer the martingale property from $\mathbb F$ to $\mathbb G$ known as the immersion property or H-hypothesis, see \cite{bremaud1978changes,aksamit2016projections}. The first fundamental hypothesis is to set the existence of a (conditional) density for the default time with a certain property. 

\begin{hypothesis}[$\mathcal P_0$-Density Hypothesis]\label{densityhyp}
There exists a map
\[
    \gamma : \Omega \times \mathbb{R}_+ \times \mathbb{R}_+ \longrightarrow \mathbb{R}_+, 
    \qquad (\omega, t, u) \longmapsto \gamma(t,u)(\omega),
\] independent of the choice of $\mathbb P\in \mathcal P_0$,
satisfying the following conditions

\begin{enumerate}[label=(\roman*)]

    \item For each fixed $u \geq 0$, the process $t \mapsto \gamma(t,u)$ is 
    $(\mathcal{F}_t)$-adapted. More precisely, for each $t \geq 0$, the map 
$   (\omega, u) \mapsto \gamma(t,u)(\omega)$
    is $\mathcal{F}_t \otimes \mathcal{B}(\mathbb{R}_+)$-measurable.
    \item 
    For every $\mathbb{P} \in \mathcal{P}_0$, $t,x \geq 0$,
    \[
        \mathbb{P}(\tau > x \mid \mathcal{F}_t) 
        = \int_x^\infty \gamma(t,u)\,du, 
        \text{ and in particular }
        \int_0^\infty \gamma(t,u)\,du = 1,
        \; \mathbb{P}\text{-a.s. for all } \mathbb{P} \in \mathcal{P}_0.
    \]

    \item 
    For all $0 \leq u \leq t$,
    \[
        \gamma(t, u) = \gamma(u, u),
        \qquad \mathbb{P}\text{-a.s. for all } \mathbb{P} \in \mathcal{P}_0.
    \]

\end{enumerate}
\end{hypothesis}

\begin{remark}\label{rem:agglambda} The $\mathcal P_0$-density Hypothesis~\ref{densityhyp} is
\emph{stronger} than the classical density hypothesis see \cite{jacod2006grossissement,jeulin2006grossissements,el2010happens} under a fixed
$\mathbb P$ since the same kernel $\gamma$ is required to represent the
$\mathcal F_t$-conditional law of $\tau$ simultaneously under
every $\mathbb P\in\mathcal P_0$.
This distinction is what enables the aggregation of the intensity
$\lambda$ and of the $U$-process across the non-dominated family
$\mathcal P_0$, leading to the well-posedness of the
auxiliary 2BSDE construction
(see Lemma~\ref{thm:auxiliaryBSDE}) in our setting. 
In particular, the 2BSDE minimality condition
$\mathrm{ess\,inf}_{\mathbb P'} \mathbb E^{\mathbb P'}[K^{\mathbb P'}_{T\wedge\tau}-K^{\mathbb P'}_{t\wedge\tau}\mid\mathcal G_t]=0$
would then become ill-defined: the intensity changes with~$\mathbb P'$ and the compensated jump would not be aggregated as emphasized in \cite{possamai2025mind}. By requiring the same conditional density $\gamma$ to work simultaneously for every $\mathbb P\in\mathcal P_0$. This stability across the model class ensures that the enlargement procedure, the immersion property, and the intensity process below are compatible with the non-dominated family of probabilities $\mathcal P_0$.

\end{remark}

\noindent \textbf{Examples.}
\begin{enumerate}
    \item $\mathcal P_0$ is composed of equivalent probability measures (absolutely continuous with respect to a reference measure, admitting a Radon-Nikodym derivative). Inspired by the weak formulation of stochastic control, where only the drift is controlled and not the volatility, we define an initial probability measure $\mathbb P^0$ under which 
    \[X_t= x+\int_0^t\sigma(s,X_s) dW_s,\]and $\mathcal P_0$ is composed by all the probability $\mathbb P^\alpha$ such that 
    \[\frac{d\mathbb P^\alpha}{d\mathbb P^0}=Z^{\alpha}_T:=\mathcal E\Big(\int_0^T \mu(t,X_t,\alpha_t)\sigma^{-1}(t,X_t)dW_t\Big),\] such that $\mathbb E^0[Z_T]=1$. Thus,
    \[\mathbb P^{\alpha}(\tau\geq x|\mathcal F_t)=\mathbb E^\alpha [\mathbf 1_{\tau\geq x}|\mathcal F_t]= \mathbb E^0 [\frac{Z^\alpha_T}{Z^\alpha_t}\mathbf 1_{\tau\geq x}|\mathcal F_t]= \mathbb P( \tau\geq x|\mathcal F_t),\]
since $Z_T^\alpha/Z_t^\alpha$ is independent of $\mathcal F_t$.
So, if we know assume that $\tau$ has a $\mathbb P^0$-conditional density, we have that
\[
    \mathbb P^0(\tau \geq x | \mathcal F_t) = \int_x^\infty \gamma(t, u) du = \mathbb P^\alpha(\tau \geq x | \mathcal F_t)
\]
so that the density is the same across probabilities.

\item $\tau$ is independent of $\mathbb F$ with a distribution invariant with any probability in $\mathcal P_0$, that is for any $\mathbb P,\tilde{\mathbb P}\in \mathcal P_0$ and $\phi^{\mathbb P}(u)= \phi^{\tilde{\mathbb P}}(u)$ for any $u>0$, where $\phi^{\mathbb P},\phi^{\tilde{\mathbb P}}$ denotes the characteristic functions of $\tau$ under the probabilities $\mathbb P$ and $\tilde{\mathbb P}$ respectively. Consequently, the existence of a density $\gamma$ is invariant by the probability measure $\mathbb P$ or $\tilde{\mathbb P}$.

As a concrete example, if we assume that $\tau$ is exponentially distributed with rate $\theta>0$ and
independent of $\mathbb F$, with the same distribution under every
$\mathbb P\in\mathcal P_0$.  Then, for $x\ge 0$,
\[
    \mathbb P(\tau>x\mid\mathcal F_t) \;=\; \mathbb P(\tau>x) \;=\; e^{-\theta x},
\]
which gives the conditional density
\[
    \gamma(t,u) \;=\; \theta\,e^{-\theta u}\,\ind_{\{u\ge 0\}},
\]
and the intensity
\[
    \lambda_t \;=\; \frac{\gamma(t,t)}{\mathbb P(\tau>t\mid\mathcal F_t)}
                 \;=\; \frac{\theta\,e^{-\theta t}}{e^{-\theta t}} \;=\; \theta.
\]
The cumulative-hazard process is
\[
    \Lambda_t \;=\; \int_0^t \lambda_s\,ds \;=\; \theta\,t,
\]
and one checks
\[
    \mathbb P(\tau>t\mid\mathcal F_t) \;=\; e^{-\Lambda_t} \;=\; e^{-\theta t},
\]
recovering the survival function of $\tau$, as expected from
\cite[Proposition~4.4]{el2010happens}.  Note in particular that $\lambda$
is constant, and the
$\hat p$-integrability condition \eqref{eq:intlambda} is trivially
satisfied for every $\hat p>1$.  Moreover, the support of $\tau$ is
$[0,\infty)$, so $\mathbb P(\tau\in[0,T])=1-e^{-\theta T}<1$ as required
in Section~\ref{sec:math}.
\end{enumerate}

As a consequence of this assumption, see for example \cite{bremaud1978changes, el2010happens}, any $(\mathbb P,\F)$-martingale is also a $(\mathbb P,\G)$-martingale. Furthermore, still under the density hypothesis, the process $H$ admits an absolutely continuous compensator, \textit{i.e.,} there exists a non-negative $\G$-predictable process $\lambda^{\G}$, such that the compensated process $M$ defined by
\[\tilde H_t := H_t - \int_{0}^t \lambda_s^{\G} ds
\]
is a $\G$-martingale. The compensator vanishes after time $\tau$ (therefore $\lambda^{\G}_t := \lambda_t \ind_{t \leq \tau}$) and 
\[
    \lambda_t := \frac{\gamma(t,t)}{\p(\tau > t | \mathcal F_t)}
\]
is an $\F$-predictable process. For a complete and deeper discussion on the properties of enlarged filtration, we refer the reader to \cite{guo2008intensity}. As a consequence of \cite[Proposition 4.4]{el2010happens}
\begin{equation}\label{eq:cond-survival-doleans}
\mathbb P(\tau>t|\mathcal F_t)=e^{-\Lambda_t},\; \Lambda_t:=\int_0^t \lambda_s ds.
\end{equation}
We now turn to the integrability of the process $\lambda$ and the support of the default time $\tau$. Denoting by $\mathcal{T}(\mathbb{A})$ the set of $\mathbb{A}$-stopping times (so we will have $\mathcal{T}(\F)$ or $\mathcal{T}(\G)$), we assume that there exists a constant $\hat p>1$ such that
\[\esssup_{\rho \in \mathcal{T}(\mathcal G)} \E^\mathbb P\Bigl[\int_{\rho}^T |\lambda_s|^{\hat p} ds \, | \,  \mathcal G_{\rho}\Bigl] < + \infty.\]
As a direct consequence of the tower property and since $\F \subseteq \G$, we get 
\begin{equation}\label{eq:intlambda}
     \esssup_{\rho \in \mathcal{T}(\G)} \E^{\mathbb P}\Bigl[\int_{\rho}^T |\lambda_s|^{\hat p} ds \, | \, \mathcal F_{\rho}\Bigl] < + \infty. 
\end{equation} 
Consequently, $\mathbb P(\tau\in [0,T])<1$, the support of $\tau$ strictly contains $[0,T]$.\\

\begin{remark}
Alternatively, one could relax the integrability conditions for a bounded default time supported on the interval $[0,T]$ at time $T$. This would lead to a second-order backward stochastic differential equation with a degenerate driver. The analysis of such an extension is closely related to the work of \cite{jeanblanc2015utility, jeanblanc2014note}. Note that the existence, uniqueness, and comparison results for 2BSDEs may benefit from the well-established theory of standard BSDEs in these references. Furthermore, one of the more recent applications of this methodology is inspired by the cryptocurrency market or cyber risk. The technical study of 2BSDEs with singular drivers is left for future research. 
\end{remark}

\subsection{Spaces} Let $t\in [0,T]$ and $\omega\in \Omega$ be fixed. We set $\mathbb X:=(\mathcal X_t)_{t\leq T}$ an arbitrary filtration. We denote by $\mathbb X_{\mathbb P}$ the augmented filtration with respect to $\p\in \mathcal P_0$. For any $\chi\in (0,\infty)$ we define
\begin{itemize}
    \item $L^2_{t,\omega,\chi}(\mathbb X)$ as the space of random variables $\xi$ such that 
    \[ \|\xi\|_{L^2_{t,\omega,\chi}(\mathbb X)}^2:=\sup_{\mathbb P\in \mathcal P(t,\omega)}\mathbb E^{\mathbb P}[e^{\chi \Lambda_t}|\xi|^2],\]

    \item $\mathbb S^2_{t,\omega,\chi}(\mathbb X)$ as the set of $\mathbb X-$progressively measurable process $Y$ with $\mathcal P(t,\omega)$ quasi surely c\`adl\`ag path on $[t,T]$ such that
    \[\|Y\|^2_{\mathbb S^2_{t,\omega,\chi}}:=\sup_{\mathbb P\in \mathcal P(t,\omega)}   \|Y\|^2_{\mathbb S^2_{t,\omega,\chi}(\mathbb P)}<\infty,\] where
    \[\|Y\|^2_{\mathbb S^2_{t,\omega,\chi}(\mathbb P)}:=\mathbb E^{\mathbb P}[\sup_{t\leq s\leq T} e^{\chi \Lambda_s}|Y_s|^2]<\infty.\]
   We similarly defined $\mathbb S^2_{t,\omega,\chi}(\mathbb X,\mathbb P)$ for a measure $\mathbb P$ fixed as the set of process $Y$ such that $\|Y\|^2_{\mathbb S^2_{t,\omega,\chi}(\mathbb P)}<\infty$.
    
    \item $\mathbb H^2_{t,\omega,\chi}(\mathbb X)$ as the set of $\mathbb X-$predictable $\mathbb R^d-$ valued processes such that
    \[\|Z\|^2_{\mathbb H^2_{t,\omega,\chi}}:= \sup_{\mathbb P\in \mathcal P(t,\omega)} \|Z\|^2_{\mathbb H^2_{t,\omega,\chi}(\mathbb P)}<\infty,\]
    where 
    \[\|Z\|^2_{\mathbb H^2_{t,\omega,\chi}(\mathbb P)}:=\mathbb E^{\mathbb P}\Big[ \int_t^T e^{\chi \Lambda_s}\|\hat{\sigma}_s^\frac12 Z_s\|^2ds\Big]<\infty.\]
    We similarly defined $\mathbb H^2_{t,\omega,\chi}(\mathbb X,\mathbb P)$ for a measure $\mathbb P$ fixed as the set of process $Z$ such that $\|Z\|_{\mathbb H^2_{t,\omega,\chi}(\mathbb P)}<\infty$.
  
    \item  $\mathbb I^2_{t,\omega}(\mathbb X,\mathbb P)$ as the set of $\mathbb X-$predictable $\mathbb R-$ valued processes $K$ with $\mathbb P\text{-a.s.}$ c\`adl\`ag and nondecreasing paths on $[t,T]$ with $K_t=0$ such that
    \[\|K\|^2_{\mathbb I^2_{t,\omega}(\mathbb X,\mathbb P)}:= \mathbb E^{\mathbb P}[K_T^2]<\infty.\]
    We say that a family $(K^\mathbb P)_{\mathbb P\in \mathcal P(t,\omega)}$ belongs to $\mathbb I^2_{t,\omega}((\mathbb X_\mathbb P)_{\mathbb P\in \mathcal P(t,\omega)})$ if $K^\mathbb P\in \mathbb I^2_{t,\omega}(\mathbb X,\mathbb P)$ and 
    \[\sup_{\mathbb P\in \mathcal P(t,\omega)} \|K^\mathbb P\|_{\mathbb I^2_{t,\omega}(\mathbb P)}<\infty.\]
    \item  $\mathbb M_{t,\omega,\chi}(\mathbb X,\mathbb P)$ denotes the space of all $(\mathbb X,\mathbb P)$-optional martingales $M$ with $\mathbb P$-a.s. c\`adl\`ag paths on $[t,T]$\, with $M_t=0$, $\mathbb P\text{-a.s.}$ and

\[
\| M \|^2_{\mathbb M^2_{t,\omega,\chi}(\mathbb P)} := \mathbb{E}^{\mathbb P}\!\left[ \int_0^T e^{\chi \Lambda_s} d[M]_s\right] < +\infty.
\]

We say that a family $(M^\mathbb P)_{\mathbb P\in \mathcal P(t,\omega)}$ belongs to $\mathbb M_{t,\omega,\chi}((\mathbb X_\mathbb P)_{\mathbb P\in \mathcal P(t,\omega)})$ if $M^\mathbb P\in \mathbb M_{t,\omega,\chi}(\mathbb X,\mathbb P)$ and 
    \[\sup_{\mathbb P\in \mathcal P(t,\omega)} \|M^\mathbb P\|_{\mathbb M^p_{t,\omega,\chi}(\mathbb P)}<\infty.\]
  \item $\mathbb J^2_{t,\omega,\chi}(\mathbb X)$ as the set of $\mathbb X-$predictable $\mathbb R-$ valued processes $U$ such that
    \[\|U\|^2_{\mathbb J^2_{t,\omega,\chi}}:= \sup_{\mathbb P\in \mathcal P(t,\omega)} \mathbb E^{\mathbb P}\Big[ \int_t^T e^{\chi (\Lambda_s-\Lambda_t)}|U_s|^2 ds\Big]<\infty.\]
\end{itemize}
For the sake of simplicity, we omit the dependence with respect to $\omega$ when $t=0$ in the previous definitions and write, for example, $L^2_{0,\chi}(\mathbb X)$ for $t=0$. 

\subsection{Formulation and Definition}
We recall that the input of a 2BSDE are a $\mathcal G_{\tau\wedge T}-$measurable terminal condition $\xi$ and a generator function $f$ from $[0,T]\times \Omega\times \mathbb R\times\mathbb R^d\times \mathbb R\times \mathcal S_d^{\geq  0}\times \mathbb R^d\times \mathbb R^+ \to \mathbb R$. We assume that the following conditions are enforced in this work.
\begin{assumption}\label{assumption:L} For every $(t,\omega,y,y',z,z',u,u',\hat{\sigma},m,\lambda)\in[0,T]\times \Omega\times \mathbb R\times \mathbb R\times \mathbb R^d\times \mathbb R^d\times \mathcal S_d^{\geq 0}\times \mathbb R^d\times \mathbb R^+$ there exists constants $C,\hat{\chi}>0$ such that

\begin{align*}
    (i)&\quad |f(s,\omega,y,z,u,\hat{\sigma}, m,\lambda)-f(s,\omega,y',z',u',\hat{\sigma}, m,\lambda)|\leq C\big(|y-y'|+\|\hat{\sigma}^{\frac12}(z-z')\| +\lambda |u-u'|\big)\\
    (ii)&\quad \sup_{\p\in \mathcal P_0} \mathbb E^{\mathbb P}\Big[\int_0^{T} e^{-\hat{\chi} \Lambda_s}|f(s,X_{\cdot\wedge s},0,0,0,\hat{\sigma}_s,\mu_s^{\mathbb P},\lambda_s^{\mathbb P})|^2\,ds\Big]<\infty.
\end{align*}
    
\end{assumption}

We define for the sake of simplicity
\[F^\mathbb P_s(y,z,u):=f(s,X_{\cdot\wedge s},y,z,u,\hat{\sigma}_s, \mu^\mathbb P_s,\lambda_s). \]
We now recall the result of \cite[Lemma 3.1]{jeanblanc2015utility}.

\begin{lemma}
    Any random variable $\xi-\mathcal G_{T\wedge \tau}$ measurable can be decomposed as follows

    \[\xi=\xi^{b}\mathbf 1_{T<\tau}+\xi^{a}_\tau\mathbf 1_{T\geq\tau},\]

    where $\xi^b$ is an $\mathcal F_T-$measurable random variable and $\xi^a$ is an $\mathbb F-$predictable process. 
\end{lemma}
We consider the following 2BSDE with erratic horizon
\begin{align}
\label{2bsde:eq} Y_t&=\xi-\int_{t\wedge \tau}^{T\wedge \tau}F_s^{\mathbb P}(Y_s,\hat{\sigma}_s^\frac12 Z_s,U_s)ds-\int_{t\wedge \tau}^{T\wedge \tau} Z_s\cdot dX_s^{c,\mathbb P}-\int_{t\wedge \tau}^{T\wedge \tau} dM_s^\mathbb P\\
\nonumber &+K_{T\wedge \tau}-K_{t\wedge \tau}- \int_{t\wedge \tau}^{T\wedge \tau} U_s dH_s,\text{ for all }0\leq t\leq T,\; \mathbb P\text{-a.s.}, \text{ for every }\p\in \mathcal P_0.
\end{align}

We recall from \cite{jeulin2006grossissements} that
\[F^{\mathbb P}_t(\cdot)\mathbf 1_{t\leq \tau}=F^{\mathbb P,b}_t(\cdot)\mathbf 1_{t\leq \tau},  \]
where $F^{\mathbb P,b}:[0,T]\times \Omega\times \mathbb R\times \mathbb R^d\times \mathbb R\longrightarrow \mathbb R$ is $\mathbb F-$progressively measurable.

\begin{definition}[2BSDE with erratic horizon]\label{def:2bsde} The tuple $(Y,Z,U, (K^{\mathbb P})_{\p\in \mathcal P_0},(M^{\mathbb P})_{\mathbb P\in \mathcal P_0})\addd{}$ is solution to the 2BSDE with erratic horizon driven by $F^\mathbb P$ with terminal condition $\xi$ if 
\begin{enumerate}
    \item $(Y,Z,U, (K^{\mathbb P})_{\p\in \mathcal P_0},(M^{\mathbb P})_{\mathbb P\in \mathcal P_0})\in  \mathbb S^2_{0,\hat{\chi}}(\mathbb G^{\mathcal P_0+})\times \mathbb H^2_{0,\hat{\chi}}(\mathbb G^{\mathcal P_0+})\times \mathbb J^{2}_{0,\hat{\chi}}(\mathbb G^{\mathcal P_0+}) \times \mathbb I^2_{0}((\mathbb G^{\mathbb P +})_{\p\in \mathcal P_0})\times \mathbb M_{0,\hat{\chi}}((\mathbb G^{\mathbb P +})_{\p\in \mathcal P_0})$ satisfies \eqref{2bsde:eq} with $Y_T=\xi$;
    \item the family $(K^\mathbb P)_{\p\in \mathcal P_0}$ satisfies the following minimality condition
    \[0=\essinf_{\mathbb P'\in \mathcal P_0(t,\mathbb P, \mathbb G_+)}\mathbb E^{\mathbb P'}[K_T^{\mathbb P'}-K_t^{\mathbb P'} | \mathcal G_t^{\mathbb P',+}],\; 0\leq t\leq T\; \mathbb P\text{-a.s.},\; \forall \p\in \mathcal P_0.\]
\end{enumerate}
\end{definition}
To study the existence of a solution to this second-order BSDE, we introduce an auxiliary BSDE defined by

\begin{align}
\label{BSDE:eq} \mathcal Y^\mathbb P_t&=\xi^{b}-\int_{t}^{T}F_s^{\mathbb P,b}(\mathcal Y^{\mathbb P}_s,\hat{\sigma}_s^\frac12 \mathcal Z^{\mathbb P}_s,\xi_s^{a}-\mathcal Y^{\mathbb P}_s)\,ds-\int_{t}^{T} \mathcal Z^{\mathbb P}_s\cdot dX_s^{c,\mathbb P}-\int_t^T d\mathcal M_s^{\mathbb P}
\text{ for all }0\leq t\leq T,\; \mathbb P\text{-a.s.}.
\end{align}

\begin{remark} Assume that $\xi$ is an $\mathcal G_{T\wedge\tau}-$measurable random variable such that $\|\xi\|_{L_{0,\hat{\chi}}^2(\mathbb G)}<\infty$, then the auxiliary BSDE \eqref{BSDE:eq} admits a unique solution $(\mathcal Y^{\mathbb P},\mathcal Z^{\mathbb P},\mathcal M^{\mathbb P})\in \mathbb S^2_{t,\omega,\hat{\chi}}(\mathbb F,\mathbb P)\times \mathbb H^2_{t,\omega,\hat{\chi}}(\mathbb F,\mathbb P)\times \mathbb M_{t,\omega,\hat{\chi}}(\mathbb F,\mathbb P) $, see for example \cite{el1997general,papapantoleon2018existence}, \cite[Assumption $(sL^{p,\beta})$]{mastrolia2018density}. 
\end{remark}

\section{Existence, uniqueness and comparison}

\label{sec:2BSDE}

\subsection{Existence via auxiliary 2BSDE}
\label{sec:existence-thm-detail}

We consider the following auxiliary 2BSDE

\begin{align}
\nonumber &Y^b_t=\xi^{b}-\int_{t}^{T}F_s^{\mathbb P,b}(Y^{b}_s,\hat{\sigma}_s^\frac12  Z^{b}_s,\xi_s^{a}-Y^{b}_s)ds-\Big(\int_{t}^{T} Z^{b}_s\cdot dX_s^{c,\mathbb P}\Big)^{\mathbb P}-\int_t^T dM^{b,\mathbb P}_s+K^{b,\mathbb P}_T-K^{b,\mathbb P}_t,\\
\label{2BSDEb:eq}&\text{ for all }0\leq t\leq T,\; \mathbb P\text{-a.s.},\; \forall \p\in \mathcal P_0.
\end{align}

\begin{lemma}\label{thm:auxiliaryBSDE}

Let Assumption \ref{assumption:L}  be satisfied. Let $\xi$ be an $\mathcal G_{T\wedge \tau}-$measurable random variable such that $\|\xi\|_{L_{0,\hat{\chi}}^2(\mathbb G)}<\infty$. The auxiliary 2BSDE \eqref{2BSDEb:eq} admits a unique solution $(Y^b,Z^b ,(M^{b,\mathbb P})_{\p\in \mathcal P_0},(K^{b,\mathbb P})_{\p\in \mathcal P_0})\in \mathbb S^2_{0,\hat\chi}(\mathbb F)\times \mathbb H^2_{0,\hat\chi}(\mathbb F) \times \mathbb M_{0,\hat\chi}((\mathbb F^{\mathbb P +})_{\p\in \mathcal P_0})\times \mathbb I^2_{0,\hat\chi}((\mathbb F^{\mathbb P +})_{\p\in \mathcal P_0})$ such that
\begin{align*}
Y_t^b&= \esssup_{\mathbb P'}  \mathcal Y_t^{\mathbb P'},\quad 
Z^b_t=\hat{\sigma}_t^{-1}\frac{d\langle Y^b,X^{c}\rangle_t}{dt},\\
M^{b,\mathbb P}_t-K_t^{b,\mathbb P}&=Y_t^b-Y_0^b+\int_0^t F_s^{\mathbb P,b}(Y_s^b,\hat{\sigma}_s^{\frac12}Z_s^b,\xi_s^{a}-Y_s^b)ds+\int_0^t Z_s^b\cdot dX_s^{c,\mathbb P},
\end{align*} 
where $(\mathcal Y^{\mathbb P},\mathcal Z^{\mathbb P},\mathcal M^{\mathbb P})$ is the unique solution to \eqref{BSDE:eq}.
\end{lemma}
\begin{proof}
Note that Assumption \ref{assumption:L} directly ensure that \cite[Assumption 2.20 (ii)]{possamai2025mind} is satisfied. Moreover, for any $(t,y,\tilde y,z,\tilde z,a)\in [0,T]\times \mathbb R\times \mathbb R\times \mathbb R^d\times \mathbb R^d\times \mathbb S_+^d$, we have
\begin{equation*}
\sup_{\mathbb P}\mathbb E^{\mathbb P}\!\left[e^{\hat\chi\,\Lambda_T}|\xi|^2 + \int_0^T e^{\hat\chi(\Lambda_T-\Lambda_r)}|F^{\mathbb P,b}_r(0,0,\xi_s^{a})|^2 dr\right]<\infty.
\end{equation*}

By taking $r^2_t(\omega):=C(1+\lambda_t)^2\geq C$ we directly get  \cite[Assumption 2.20 (iv)]{possamai2025mind}. Then, (i)-(iii)-(v) are automatically satisfied since we are taking a fixed horizon $T$ for the auxiliary 2BSDE and $\Lambda_t$ is absolutely continuous. The result is thus a consequence of \cite[Theorem 3.6 and Corollary 3.7]{possamai2025mind}.
\end{proof}

\begin{remark}
    Note that the universal measurability of the auxiliary 2BSDE is directly derived from \cite[Theorem~3.1 and Section~3.1, ``Construction of the value function and its measurability'']{possamai2025mind} and the condition enforced on the set $\mathcal P_0$.
\end{remark}
The following corollary enforces the aggregation of the stochastic integral with respect to $X^{c,\mathbb P}$ and the process $K^\mathbb P-M^\mathbb P$ can be aggregated as soon as \textcolor{blue}{$\mu^\mathbb P=0$.}
\begin{corollary}
    Assume that $\mu^{\mathbb P}=0$ then for any $\p\in \mathcal P_0$
\[
\int_0^\cdot Z_s \cdot dX_s = \int_0^\cdot Z_s \cdot dX^{c,\mathbb P}_s \quad \mathbb P\text{-a.s.},
\]
and there exists an $\mathbb F^{\mathcal P_0+}-$ progressively measurable process
which aggregates $K^\mathbb P-M^{\mathbb P}$.
\end{corollary}

\begin{proof}
    This is a consequence of the $\mathcal P_0-$density assumption and \cite[Remark 4.1]{possamai2018stochastic}.
\end{proof}
The following theorem provides the construction of a solution to the 2BSDE \eqref{2bsde:eq} with random default and jump from the auxiliary 2BSDE \eqref{2BSDEb:eq} without jump. The key ideas is to build the solution to \eqref{2bsde:eq} under three scenarios: 
\begin{itemize}
    \item $\{\tau>T\}$, that is  the default occurs after $T$. In this case, the solution to the 2BSDE does not need the default jumps and matches with the solution to the auxiliary 2BSDE;
    \item $\{\tau\in(t,T]\}$, that is when the default occurs in the future but before maturity. This is then reduce to a 2BSDE with random horizon and terminal condition $\xi^a_\tau$ since the jump occurs before $T$. 
    \item $\{\tau\leq t\}$ when the defaults has already occurred and the $Y$ component is frozen at $\xi^a_\tau$. 
\end{itemize}

\begin{theorem}\label{thm:decompbsde}
     Assume that $\xi$ is an $\mathcal G_{T\wedge \tau}-$measurable random variable such that $\|\xi\|_{L_{0,\hat{\chi}}^{2}}<\infty$ and \eqref{BSDE:eq} admits a unique solution $(\mathcal Y^\mathbb P,\mathcal Z^\mathbb P,\mathcal M^\mathbb P ) \in  \mathbb S^2_{0,\hat\chi}(\mathbb F,\mathbb P)\times \mathbb H^2_{0,\hat\chi}(\mathbb F,\mathbb P)\times \mathbb M_{0,\hat\chi}(\mathbb F,\mathbb P)$ for any $\p\in \mathcal P_0$. Then, the 2BSDE \eqref{2bsde:eq} admits a solution $(Y,Z,U,M,K)\in  \mathbb S^2_{0,\hat\chi}(\mathbb G^{\mathcal P_0+})\times \mathbb H^2_{0,\hat\chi}(\mathbb G^{\mathcal P_0+})\times \mathbb J^{2}_{0,\hat\chi}(\mathbb G^{\mathcal P_0+}) \times \mathbb I^2_{0,\hat\chi}((\mathbb G^{\mathbb P +})_{\p\in \mathcal P_0})\times \mathbb M_{0,\hat\chi}((\mathbb G^{\mathbb P +})_{\p\in \mathcal P_0})$ given by

\begin{align}
\nonumber &Y_t= Y_t^b \mathbf 1_{t<\tau} + \xi^{a}_\tau \mathbf 1_{t\geq \tau},\, Z_t=Z_t^b \mathbf 1_{t<\tau},\; M_t^{\mathbb P}=M_t^{b,\mathbb P}  \mathbf 1_{t<\tau},\; K_t^{\mathbb P}= K_t^{b,\mathbb P}  \mathbf 1_{t<\tau},\; U_t^{\mathbb P}=(\xi^{a}_t-Y_t^b)\mathbf 1_{t<\tau},\\
\label{2bsdedecomposition}& \forall t<T,\; \mathbb P\text{-a.s.}, \p\in \mathcal P_0
\end{align}

\end{theorem}

\begin{proof} The proof extends \cite[Theorem 4.3]{kharroubi2013mean} to 2BSDE. We divide this proof into four main steps: first, we recall that the auxiliary 2BSDE \eqref{2BSDEb:eq} has a solution, then we verify the decomposition and measurability of the solution and check that \eqref{2bsdedecomposition} is indeed solutions to \eqref{2bsde:eq}, then we show the integrability of each component. Finally, we check the minimality condition for the $K$ process.\vspace{0.5em}

\textit{Step 1. 2BSDE \eqref{2BSDEb:eq}.} From Theorem \ref{thm:auxiliaryBSDE} we deduce that there exists a unique solution \eqref{2BSDEb:eq} such that $(Y^b,Z^b,M^b,K^b)\in  \mathbb S^2_{0,\hat\chi}(\mathbb F^{\mathcal P_0+})\times \mathbb H^2_{0,\hat\chi}(\mathbb F^{\mathcal P_0+})\times \mathbb I^2_{0,\hat\chi}((\mathbb F^{\mathbb P +})_{\p\in \mathcal P_0})\times \mathbb M_{0,\hat\chi}((\mathbb F^{\mathbb P +})_{\p\in \mathcal P_0})$.\vspace{0.5em}

 \textit{Step 2. Existence, measurability, and decomposition.} We prove that for any time $t\in [0,T]$ and any probability $\p\in \mathcal P_0$, the processes $(Y,Z,M^\mathbb P,K^\mathbb P,U^\mathbb P)$ defined by \eqref{2bsdedecomposition} satisfies the primal 2BSDE \eqref{2bsde:eq}. First, from Theorem \ref{thm:auxiliaryBSDE}, we directly deduce that $Y$ is a c\`adl\`ag $\mathbb G^{\mathcal P_0,+}-$adapted processes. The $\mathbb G^{\mathcal P_0,+}-$measurability of $Z,U,K,M$ follows from the $\mathbb F-$measurability of $Z^b,Y^b,K^b,M^b$ and the continuity of $Y^b,\xi^{a}$. We denote by $A^+=\{\tau>T\}$, $A=\{\tau\in (t,T]\}$ and $A_-=\{\tau\leq t\}$. 
 \begin{itemize}
 \item On $A^+$ such that $\mathbb P(A^+)>0$, we have $Y_t=Y^b_t$, $Z^b_t=Z_t$, $M
    _t^\mathbb P=M
    _t^{b,\mathbb P}$, $K_t^{\mathbb P}=K_t^{b,\mathbb P}$ and $U_t^\mathbb P=\xi_t^{a}-Y_t^b$. Therefore, on the set $A^+$, we have
    \begin{align*}
  Y_t&=Y_t^b=\xi^{b}+\int_t^T F
    _s^{\mathbb P,b}(Y^b_s,Z^b_s,\xi_s^{a}-Y_s)ds - \Big(\int_t^T Z^b_s\cdot dX_s^{c,\mathbb P}\Big)^{\mathbb P}-\int_t^T dM_s^{b,\mathbb P}+\int_t^T dK_s^{b,\mathbb P},\\
    &=\xi+\int_t^T F
    _s^{\mathbb P}(Y_s,Z_s,U_s)ds - \Big(\int_t^T Z_s\cdot dX_s^{c,\mathbb P}\Big)^{\mathbb P}-\int_t^T dM_s^{\mathbb P}+\int_t^T dK_s^{\mathbb P}, t\in [0,T].\end{align*}
\item On $A$ such that $\mathbb P(A)>0$, we have
    \begin{align*}
  Y_t&=Y^b_\tau+\int_t^{\tau} F
    _s^{\mathbb P}(Y_s,Z_s,\xi_s^{a}-Y_s)ds - \Big(\int_t^{\tau} Z^b_s\cdot dX_s^{c,\mathbb P}\Big)^{\mathbb P}-\int_t^{\tau} dM_s^{b,\mathbb P}+\int_t^{\tau} dK_s^{b,\mathbb P},\\
    &=\xi^{a}_\tau+\int_t^{\tau} F
    _s^{\mathbb P}(Y_s,Z_s,\xi_s^{a}-Y_s)ds - \Big(\int_t^{\tau} Z^b_s\cdot dX_s^{c,\mathbb P}\Big)^{\mathbb P}-\int_t^{\tau} dM_s^{b,\mathbb P}+\int_t^{\tau} dK_s^{b,\mathbb P}\\
    &\quad -(\xi_\tau^{a}-Y_\tau),\\
    &=\xi^{a}_\tau+\int_t^{\tau} F
    _s^{\mathbb P}(Y_s,Z_s,\xi_s^{a}-Y_s)ds - \Big(\int_t^{\tau} Z^b_s\cdot dX_s^{c,\mathbb P}\Big)^{\mathbb P}-\int_t^{\tau} dM_s^{b,\mathbb P}+\int_t^{\tau} dK_s^{b,\mathbb P}\\
    &\quad -\int_t^\tau U^\mathbb P_s dH_s,\; t\in [0,T].
    \end{align*}

\item On $A_-$ such that $\mathbb P(A_-)>0$, we note that 
    \begin{align*}
  Y_t&=\xi_{\tau}^{a}\\
  &=\xi_{\tau}^{a}+\int_{t\wedge \tau}^{T\wedge\tau} F
    _s^{\mathbb P}(Y_s,Z_s,\xi_s^{a}-Y_s)ds - \Big(\int_{t\wedge \tau}^{T\wedge\tau}Z^b_s\cdot dX_s^{c,\mathbb P}\Big)^{\mathbb P}\\
    &\quad -\int_{t\wedge \tau}^{T\wedge\tau}dM_s^{b,\mathbb P}+\int_{t\wedge \tau}^{T\wedge\tau}dK_s^{b,\mathbb P},\; t\in [0,T].
    \end{align*}
\end{itemize}
    In all cases, we have seen that for any $\p\in \mathcal P_0$, the processes $(Y,Z,M^\mathbb P,K^\mathbb P,U^\mathbb P)$ defined by \eqref{2bsdedecomposition} is solution to \eqref{2bsde:eq}.\vspace{0.5em}
    
\textit{Step 3. Integrability.} We note that the integrability of $Y,Z,M,K$ is inherited from the integrability of $Y^b,Z^b,K^b,M^b$ and $\xi^a$. Note that 

\[\mathbb E^\mathbb P\Big[\int_0^T e^{\hat\chi \Lambda_t} |U_t|^2 dt \Big] \leq 2\Big(\mathbb E^\mathbb P\Big[\int_0^T e^{\hat\chi \Lambda_t} |\xi^a_t|^2 dt \Big]+ T\mathbb E^\mathbb P\Big[\sup_{t\leq T} e^{\hat\chi \Lambda_t} |Y^b_t|^2 \Big] \Big)<\infty, \]
from Assumption \ref{assumption:L}(ii) and Lemma \ref{thm:auxiliaryBSDE}.\vspace{0.5em}

\textit{Step 4. Minimality condition: interaction $K$ and default time $\tau$.}
On $\{t<\tau\}$, the Hypothesis \ref{densityhyp} ensures that $\mathbb F$-adapted
increments have the same $\mathcal G^+_t$- and
$\mathcal F^+_t$-conditional expectations under each
$\mathbb P'\in\mathcal P_0$, see, e.g.,
\cite{jeulin2006grossissements,jeanblanc2009progressive,bremaud1978changes},
so the minimality of $K^{\mathbb P}$ is inherited from $K^{b,\mathbb P}$ and
Lemma~\ref{thm:auxiliaryBSDE}:
\begin{align*}
\mathrm{ess\,inf}_{\mathbb P'\in\mathcal P_0(t,\mathbb P,\mathbb G_+)}
\mathbb E^{\mathbb P'}\!\left[K^{\mathbb P'}_{T\wedge\tau}-K^{\mathbb P'}_{t\wedge\tau}
\,\big|\,\mathcal G^{\mathbb P',+}_t\right]
&=\mathrm{ess\,inf}_{\mathbb P'\in\mathcal P_0(t,\mathbb P,\mathbb F_+)}
\mathbb E^{\mathbb P'}\!\left[K^{b,\mathbb P'}_{T\wedge\tau}-K^{b,\mathbb P'}_{t\wedge\tau}
\,\big|\,\mathcal F^{\mathbb P',+}_t\right]\\
&=0 \text{ for } \omega\in A^+.
\end{align*}  On the complementary event $\{\tau \leq t\}$, both
$K^{\mathbb P'}_{T\wedge\tau}-K^{\mathbb P'}_{t\wedge\tau}=0$ and
$K^{b,\mathbb P'}_{T\wedge\tau}-K^{b,\mathbb P'}_{t\wedge\tau}=0$, so the
identity is trivial. 
\end{proof}

\begin{remark} The jump induced by the default $\tau$ does not appear in the auxiliary 2BSDE \eqref{2BSDEb:eq} by construction: the
auxiliary equation is defined on $[0,T]$ and the jump is reintroduced via
the before/after-default decomposition
\eqref{2bsdedecomposition}. The decomposition \eqref{2bsdedecomposition} is to
be read as ``$K^{\mathbb P}$ tracks $K^{b,\mathbb P}$ up to time $\tau$ and
is then frozen at its $\tau$-value''; the unambiguous non-decreasing
realization is $K^{\mathbb P}_t := K^{b,\mathbb P}_{t\wedge\tau}$, with which
all integrals appearing in \eqref{2bsde:eq} (which are themselves
stopped at $\tau$) are unchanged, and $K^{\mathbb P}_{T\wedge\tau}-K^{\mathbb P}_{t\wedge\tau}
=K^{b,\mathbb P}_{T\wedge\tau}-K^{b,\mathbb P}_{t\wedge\tau}$ pathwise.
\end{remark}

\subsection{Uniqueness of solution to 2BSDE with erratic horizon}

We define for all $0\leq t\leq T$

\begin{align}
\label{BSDE:eq:tau} \mathcal Y^{\mathbb P,\tau}_t&=\xi-\int_{t\wedge \tau}^{T\wedge \tau}F_s^{\mathbb P}(\mathcal Y^{\mathbb P,\tau}_s,\hat{\sigma}_s^\frac12 \mathcal Z^{\mathbb P,\tau}_s,\mathcal U^{\mathbb P,\tau}_s)ds-\int_{t\wedge \tau}^{T\wedge \tau} \mathcal Z^{\mathbb P,\tau}_s\cdot dX_s^{c,\mathbb P}\\
\nonumber &\quad - \int_{t\wedge \tau}^{T\wedge \tau} \mathcal U^{\mathbb P,\tau}_s dH_s-\int_{t\wedge \tau}^{T\wedge \tau} d\mathcal M_s^{\mathbb P,\tau},\; \mathbb P\text{-a.s.}.
\end{align}

\begin{lemma}\label{lemma:bsde} Assume that Assumption \ref{assumption:L} is satisfied and $\xi$ is an $\mathcal G_{T\wedge \tau}-$measurable random variable such that $\|\xi\|_{L_{0,\hat\chi}^{2}}<\infty$. Fix $\p\in \mathcal P_0$. Then there exists a unique solution $(\mathcal Y^{\mathbb P,\tau},\mathcal Z^{\mathbb P,\tau}, \mathcal U^{\mathbb P,\tau},\mathcal M^{\mathbb P,\tau})\in  \mathbb S^2_{0,\hat\chi}(\mathbb G)\times \mathbb H^2_{0,\hat\chi}(\mathbb G)\times \mathbb J^{2}_{0,\hat\chi}(\mathbb G) \times \mathbb M_{0,\hat\chi}(\mathbb G)$ to the erratic horizon BSDE \eqref{BSDE:eq:tau}.
\end{lemma}

\begin{proof} The proof is decomposed into two parts: first, we deduce the existence of a solution to BSDE \eqref{BSDE:eq:tau} from Assumption \ref{assumption:L}. Then, we turn to the uniqueness. 
\begin{itemize}
    \item[(i)] The existence of a solution is similar to \cite{kharroubi2013mean} or \cite{jeanblanc2015utility} extended to unbounded but integrable $\lambda$ processes with a Lipschitz generator and outside a utility maximization problem. It also follows the same lines as the proof of Theorem \ref{thm:decompbsde} above. We only recall the decomposition of $(\mathcal Y^{\mathbb P,\tau},\mathcal Z^{\mathbb P,\tau}, \mathcal U^{\mathbb P,\tau},\mathcal M^{\mathbb P,\tau})$ in terms of the solution $(\mathcal Y^{\mathbb P},\mathcal Z^{\mathbb P}, \mathcal M^{\mathbb P})$ to BSDE \eqref{BSDE:eq}.

    \begin{align}
\nonumber &\mathcal Y^{\mathbb P,\tau}_t= \mathcal Y^{\mathbb P}_t \mathbf 1_{t<\tau} + \xi^{a}_\tau \mathbf 1_{t\geq \tau},\, \mathcal Z^{\mathbb P,\tau}_t=\mathcal Z^{\mathbb P}_t \mathbf 1_{t<\tau},\; \mathcal M^{\mathbb P,\tau}_t=\mathcal M_t^{\mathbb P}  \mathbf 1_{t<\tau},\; \mathcal U_t^{\mathbb P,\tau}=(\xi^{a}_t-\mathcal Y_t^{\mathbb P})\mathbf 1_{t<\tau},\\
\label{bsdedecomposition}& \forall t<T,\; \mathbb P\text{-a.s.}
\end{align}

\item[(ii)] We now turn to the uniqueness under Assumption \ref{assumption:L} which follows classical linearization techniques; see, for example, \cite{el1997backward}. We consider two solutions to BSDE \eqref{BSDE:eq:tau} denoted by $(\mathcal Y^{\mathbb P,\tau},\mathcal Z^{\mathbb P,\tau}, \mathcal U^{\mathbb P,\tau},\mathcal M^{\mathbb P,\tau})$ and $(\tilde{\mathcal Y}^{\mathbb P,\tau},\tilde{\mathcal Z}^{\mathbb P,\tau}, \tilde{\mathcal U}^{\mathbb P,\tau},\tilde{\mathcal M}^{\mathbb P,\tau})$ and we set
\[\delta_t^{\mathcal V}:= \mathcal V_t^{\mathbb P,\tau}-\tilde{\mathcal V}_t^{\mathbb P,\tau},\; \mathcal V\in\{\mathcal Y,\mathcal Z,\mathcal U,\mathcal M\}.\]Then, $(\delta \mathcal Y,\delta\mathcal Z,\delta\mathcal U,\delta\mathcal M)$ is solution to

\begin{align*} \delta_t^{\mathcal Y}&=0-\int_{t\wedge \tau}^{T\wedge \tau}\delta_s^{ F^{\mathbb P}}ds-\int_{t\wedge \tau}^{T\wedge \tau} \delta_s^{\mathcal Z}\cdot dX_s^{c,\mathbb P}- \int_{t\wedge \tau}^{T\wedge \tau} \delta_s^{\mathcal U}  dH_s-\int_{t\wedge \tau}^{T\wedge \tau} d\delta_s^{\mathcal M},\; \mathbb P\text{-a.s.},
\end{align*}
where 
\begin{align*}
  \delta_s^{ F^{\mathbb P}}&:=F_s^{\mathbb P}(\mathcal Y^{\mathbb P,\tau}_s,\hat{\sigma}_s^\frac12 \mathcal Z^{\mathbb P,\tau}_s,\mathcal U^{\mathbb P,\tau}_s)- F_s^{\mathbb P}(\tilde{\mathcal Y}^{\mathbb P,\tau}_s,\hat{\sigma}_s^\frac12 \tilde{\mathcal Z}^{\mathbb P,\tau}_s,\tilde{\mathcal U}^{\mathbb P,\tau}_s)\\
  &= F_s^{\mathbb P}(\mathcal Y^{\mathbb P,\tau}_s,\hat{\sigma}_s^\frac12 \mathcal Z^{\mathbb P,\tau}_s,\mathcal U^{\mathbb P,\tau}_s)- F_s^{\mathbb P}(\tilde{\mathcal Y}^{\mathbb P,\tau}_s,\hat{\sigma}_s^\frac12 \mathcal Z^{\mathbb P,\tau}_s,\mathcal U^{\mathbb P,\tau}_s)\\
  &+ F_s^{\mathbb P}(\tilde{\mathcal Y}^{\mathbb P,\tau}_s,\hat{\sigma}_s^\frac12 \mathcal Z^{\mathbb P,\tau}_s,\mathcal U^{\mathbb P,\tau}_s)- F_s^{\mathbb P}(\tilde{\mathcal Y}^{\mathbb P,\tau}_s,\hat{\sigma}_s^\frac12 \tilde{\mathcal Z}^{\mathbb P,\tau}_s,\mathcal U^{\mathbb P,\tau}_s)\\
  &+F_s^{\mathbb P}(\tilde{\mathcal Y}^{\mathbb P,\tau}_s,\hat{\sigma}_s^\frac12 \tilde{\mathcal Z}^{\mathbb P,\tau}_s,\mathcal U^{\mathbb P,\tau}_s)-F_s^{\mathbb P}(\tilde{\mathcal Y}^{\mathbb P,\tau}_s,\hat{\sigma}_s^\frac12 \tilde{\mathcal Z}^{\mathbb P,\tau}_s,\tilde{\mathcal U}^{\mathbb P,\tau}_s)\\
  &=\theta_s \delta_s^{\mathcal Y}+\zeta_s \delta_s^{\mathcal Z} +\nu_s \lambda_s \delta_s^{\mathcal U},
\end{align*}
where $\theta$, $\zeta$, $\nu$ are bounded processes obtained from the mean-value theorem under Assumption~\ref{assumption:L}. From Girsanov theorem, changing the probability measure for both the martingales $X^{c,\mathbb P}$ and $\tilde H$ and by considering the process $\tilde D_t:= e^{-\int_0^{t\wedge \tau} \theta_s ds} \delta_t^{\mathcal Y}$, we deduce that $\tilde D_t=0$ and so $\delta_t^{\mathcal Y}=0$ for any time $t\in [0,T]$,\; $\mathbb P\text{-a.s.}$ Consequently, $\mathcal Y_t^{\mathbb P,\tau}=\tilde{\mathcal Y}_t^{\mathbb P,\tau}$ for all $t\leq T,\; \mathbb P\text{-a.s.}$ The uniqueness of the solution then follows by the uniqueness of the semimartingale representation of a process.\qedhere
\end{itemize}
\end{proof}

As a consequence of Assumption \ref{assumption:L}, \cite[Theorem 4.2]{possamai2018stochastic} or \cite{possamai2025mind} we have the following dynamic programming principle. 
\begin{proposition}[Dynamic programming]\label{DPP}Assume that $\xi$ is an $\mathcal G_{T\wedge \tau}-$measurable random variable such that $\|\xi\|_{L_{0,\hat\chi}^{2}}<\infty$. Then, any solution $(Y,Z,U,M,K)\in  \mathbb S^2_{0, \hat{\chi}}(\mathbb G^{\mathcal P_0+})\times \mathbb H^2_{0, \hat{\chi}}(\mathbb G^{\mathcal P_0+})\times \mathbb J^{2}_{0, \hat{\chi}}(\mathbb G^{\mathcal P_0+}) \times \mathbb I^2_{0, \hat{\chi}}((\mathbb G^{\mathbb P +})_{\p\in \mathcal P_0})\times \mathbb M_{0, \hat{\chi}}((\mathbb G^{\mathbb P +})_{\p\in \mathcal P_0})$ to the 2BSDE \eqref{2bsde:eq}
satisfies the following dynamic programming principle for any $0\leq t_1\leq t_2\leq T$
\[Y_{t_1}=\text{ess sup}^{\mathbb P}_{\mathbb P'\in \mathcal P(t_1,\mathbb P,\mathbb G)}\; \mathcal Y_{t_1}^{\mathbb P',\tau}(t_2,Y_{t_2}),\; \mathbb P\text{-a.s.},\]
where $\big(\mathcal Y^{\mathbb P',\tau}(t_2,Y_{t_2}),\mathcal Z^{\mathbb P',\tau}(t_2,Y_{t_2}), \mathcal U^{\mathbb P',\tau}(t_2,Y_{t_2}),\mathcal M^{\mathbb P',\tau}(t_2,Y_{t_2})\big)$ is a solution to \eqref{BSDE:eq:tau} with horizon $T=t_2$ and terminal condition $\xi=Y_{t_2}$.
\end{proposition}
Consequently, by the uniqueness of the solution to \eqref{BSDE:eq:tau} we have the following corollary.

\begin{corollary}[Uniqueness]\label{thm:lipschitz}Assume that $\xi$ is an $\mathcal G_{T\wedge \tau}-$measurable random variable such that $\|\xi\|_{L_{0,\hat\chi}^{2}}<\infty$ and Assumption \ref{assumption:L} is enforced. Then 2BSDE \eqref{2bsde:eq} has at most one solution in the sense of Definition \ref{def:2bsde}.
\end{corollary}

\begin{remark}\label{remark:decomp}
        Note that on the one hand, Theorem \ref{thm:decompbsde} provides a decomposition of the solution to \eqref{2bsde:eq} provided the existence of a solution to the auxiliary BSDE and 2BSDE \eqref{BSDE:eq} and \eqref{2BSDEb:eq}. On the other hand, under the Lipschitz Assumption \ref{assumption:L}, the DPP gives a representation of \eqref{2bsde:eq} as a supremum of the solution to the erratic horizon BSDE \eqref{BSDE:eq:tau}. The following relation proves the consistency of the solutions given the two methods: 
        
        \begin{align*}Y_t\overset{Proposition \; \ref{DPP}}{=}\esssup_{\mathbb P} \mathcal Y_t^{\mathbb P,\tau}\overset{Lemma \; \ref{lemma:bsde}}{=}\esssup_{\mathbb P} \mathcal Y_t^{\mathbb P,b}\mathbf 1_{t<\tau} +\xi^a_\tau \mathbf 1_{t>\tau}
        &\overset{Lemma \; \ref{thm:auxiliaryBSDE}}{=}Y_t^b \mathbf 1_{t<\tau} +\xi^a_\tau \mathbf 1_{t>\tau}\\
        &\overset{Theorem \; \ref{thm:decompbsde}}{=}Y_t.  
  \end{align*}
    \end{remark}

\begin{remark}\label{rkDPP}
    Alternatively, one could consider the following 2BSDE
    \begin{align}
\label{2bsde:eq-} Y_t&=\xi-\int_{t\wedge \tau}^{T\wedge \tau}F_s^{\mathbb P}(Y_s,\hat{\sigma}_s^\frac12 Z_s,U_s)ds-\int_{t\wedge \tau}^{T\wedge \tau} Z_s\cdot dX_s^{c,\mathbb P}-\int_{t\wedge \tau}^{T\wedge \tau} dM_s^\mathbb P\\
\nonumber &-(K_{T\wedge \tau}-K_{t\wedge \tau})- \int_{t\wedge \tau}^{T\wedge \tau} U_s dH_s,\text{ for all }0\leq t\leq T,\; \mathbb P\text{-a.s.}, \text{ for every }\p\in \mathcal P_0,
\end{align}
where we have to consider a minus sign for the $K$ term. In this case, the previous result still holds up to an infimum instead of a supremum and more particularly: 

\[ Y_t\overset{Proposition \; \ref{DPP}}{=}\essinf_{\mathbb P} \mathcal Y_t^{\mathbb P,\tau}\]
\end{remark}

\subsection{Comparison theorem for Lipschitz 2BSDE with erratic horizon}

\begin{theorem}[Comparison]\label{comparaisontheorem}
Let $(Y,Z,U, (K^{\mathbb P})_{\p\in \mathcal P_0},(M^{\mathbb P})_{\mathbb P\in \mathcal P_0})\addd{}$ and $(\tilde Y,\tilde Z,\tilde U, (\tilde K^{\mathbb P})_{\p\in \mathcal P_0},(\tilde M^{\mathbb P})_{\mathbb P\in \mathcal P_0})\addd{}$ be solutions to the 2BSDE \eqref{2bsde:eq} with respective terminal conditions $\xi$ and $\tilde \xi$ and generators $F^{\mathbb P}$ and $\widetilde{F^{\mathbb P}}$. Assume that $\xi^b\geq  \tilde \xi^b$,\; $\xi^{a}=\tilde{\xi}^{a}$ and $F^{\mathbb P,b}_t(y,z,u-y)\geq \widetilde{F^{\mathbb P,b}_t}(y,z,u-y) $,\; $\mathbb P\text{-a.s.}$ for any $\p\in \mathcal P_0$. Then, $Y_t\geq \tilde{Y_t},\; \mathbb P\text{-a.s.},\; \forall \p\in \mathcal P_0.$
\end{theorem}

\begin{proof}
    The proof is divided into two main parts: first, we show a comparison theorem for the auxiliary BSDE \eqref{2BSDEb:eq}, then we deduce from Remark \ref{remark:decomp}. \textcolor{blue}{The proof exploits the decomposition of the 2BSDE solution via the auxiliary BSDE, together with the freezing of its value at $\xi^a$ after the default time.}\vspace{0.5em}
    
\textit{Step 1. Comparison auxiliary BSDE. }Following the same argument as the proof of Lemma \ref{lemma:bsde} by considering 
    \[\delta_t^{\mathcal V}:= \mathcal V^{\mathbb P,b}_t-\tilde{\mathcal V}^{\mathbb P}_t,\; \mathcal V\in\{\mathcal Y,\mathcal Z,\mathcal M\}.\]
    By using the same linearization, we deduce that a comparison theorem holds true for the solutions to the auxiliary BSDE \eqref{BSDE:eq}: 
    \[\mathcal Y_t^{\mathbb P,b}\geq \widetilde{\mathcal Y}_t^{\mathbb P,b},\; t\leq T,\; \mathbb P\text{-a.s.},\; \p\in \mathcal P_0.\]

\textit{Step 2. Comparison and decomposition.} From Remark \ref{remark:decomp}, we deduce that 
    \begin{align*}
    Y_t&= \esssup_{\mathbb P} \mathcal Y_t^{\mathbb P,b}\mathbf 1_{t<\tau} +\xi^a_\tau \mathbf 1_{t>\tau}\\
    &\geq \esssup_{\mathbb P} \widetilde{\mathcal Y}_t^{\mathbb P,b}\mathbf 1_{t<\tau} +\xi^a_\tau \mathbf 1_{t>\tau}= \esssup_{\mathbb P} \widetilde{\mathcal Y}_t^{\mathbb P}\mathbf 1_{t<\tau} +\xi^a_\tau \mathbf 1_{t>\tau}=\tilde Y_t.\qedhere \end{align*}
\end{proof}

\begin{remark}
\label{sec:aggregation}
    Note that under the $\mathcal P_0-$density assumption, the process $\gamma$ and so $\lambda$ is aggregated. Existing
2BSDE-with-jumps frameworks
\cite{kazi2015second,denis2024second,kim2024representation} carry
\emph{one} jump-size process per measure $\mathbb P\in\mathcal P_0$
because the predictable compensator depends on~$\mathbb P$. Hypothesis~\ref{densityhyp}
thus allows a single $\mathbb G^{\mathcal P_0+}$-predictable~$U$ to
work for all~$\mathbb P\in\mathcal P_0$.
The aggregation of $U$ is structural and uses
Theorem~\ref{thm:decompbsde}: on $\{t<\tau\}$, the decomposition
gives
\[
U_t=\xi^a_t-Y^b_t,
\]
and both factors does not depend on the choice of $\mathbb P$ under
Hypothesis~\ref{densityhyp}. Indeed, $\xi^a$ is the $\mathbb F$-%
predictable kernel from the decomposition
$\xi=\xi^b\ind_{T<\tau}+\xi^a_\tau\ind_{T\geq\tau}$, and its
construction is driven by the $\mathbb F$-conditional law of
$\tau$ which is governed by the aggregated $\gamma$ and is
therefore the same under every $\mathbb P\in\mathcal P_0$. The
auxiliary value $Y^b_t=\esssup_{\mathbb P'}\mathcal Y^{\mathbb P'}_t$
is $\mathbb F^{\mathcal P_0+}$-measurable by
Lemma~\ref{thm:auxiliaryBSDE}, hence $\mathbb P$-independent by
construction. The product $(\xi^a_t-Y^b_t)\ind_{t<\tau}$ is then a
single $\mathbb G^{\mathcal P_0+}$-predictable process, and the
compensator of $\int U\,dH$ is the unique expression
$\int U\,\lambda\,\ind_{s\leq\tau}\,ds$ under every
$\mathbb P\in\mathcal P_0$.
\end{remark}

\subsection{Markovian 2BSDE with erratic horizon and connection with fully non linear PDE}\label{section:PDE}
In the whole section, we denote by $\mathcal C^{1,2}([0,T]\times \mathbb R^d)$ the set of maps from time and space $[0,T]\times \mathbb R^d$ into $\mathbb R$ continuously differentiable map with respect to the time component and twice continuously differentiable with respect to the space variable for any $v\in \mathcal C^{1,2}([0,T]\times \mathbb R^d)$ we denote by $\partial_t v,\nabla v,\Delta v$ its derivative in time, gradient and Laplacian respectively. In this section, we assume that the following conditions are satisfied.

\begin{assumption}We set the main Markovian assumptions for this section as follows:
   \begin{itemize}
   \item There exists a function $g:[0,T]\times \mathbb R^d\longrightarrow \mathbb R $ such that $\xi=g(T,X_T)$.
   \item  $\mu^\mathbb P=0$ so that the canonical process given by $X$ is a martingale with quadratic variation $\hat{\sigma}$. In particular, the process $W^\mathbb P$ defined by
   \[W_t^{\mathbb P}:=\int_0^t \hat{\sigma}_s^{-1/2} dX_s,\; \mathbb P\text{-a.s.}, \] is a $\mathbb P-$Brownian motion for any $\p\in \mathcal P_0$.
   \item There exists a generator $H^b:[0,T]\times \mathbb R^d\times \mathbb R\times \mathbb R^d\times \mathbb R\times D_H\longrightarrow \mathbb R$ where $D_H$ is a subset of $\mathbb R^{d\times d}$ containing $0$ such that
   \[F^b_t(x,y,z,u,\hat{\sigma}):= \sup_{\gamma\in \mathbb R^{d\times d}} \Big\{ \frac12 Tr(\hat{\sigma}\gamma)-H^b(t,x,y,z,u,\gamma)  \Big\},\; \hat{\sigma}\in \mathcal S_d^{>0}.\]
   \end{itemize}
\end{assumption}
The 2BSDE \eqref{2bsde:eq} thus becomes
\begin{align}
\nonumber Y_t&=g(T,X_T)-\int_{t\wedge \tau}^{T\wedge \tau}F_s(Y_s, Z_s,U_s,\hat{\sigma}_s)ds-\int_{t\wedge \tau}^{T\wedge \tau} Z_s\cdot dX_s +K_{T\wedge \tau}-K_{t\wedge \tau}- \int_{t\wedge \tau}^{T\wedge \tau} U_s dH_s,\\
\label{2BSDEmarkov}&\text{ for all }0\leq t\leq T,\; \mathbb P\text{-a.s.}, \text{ for every }\p\in \mathcal P_0,
\end{align}
where $F$ is defined by \[F_s(y,z,u,a):=f(s,X_{s},y,z,u,\hat{\sigma}_s,\lambda_s), \] and where $f$ satisfied the Lipschitz Assumption \ref{assumption:L}.
\begin{remark}
    Note that $F$ may not be a Markovian function of $t,X_t$ since we are not specifying any regularity assumption for the intensity process $\lambda$. 
\end{remark}
We are recalling the auxiliary 2BSDE in this 
Markovian framework

\begin{align}
\nonumber &Y^b_t=g(T,X_T)-\int_{t}^{T}F_s^{b}(Y^{b}_s, Z^{b}_s,g(s,X_s)-Y^{b}_s,\hat{\sigma}_s)ds-\int_{t}^{T} Z^{b}_s\cdot d X_s+K^{b,\mathbb P}_T-K^{b,\mathbb P}_t,\\
\label{2BSDEb:eq:markov}&\text{ for all }0\leq t\leq T,\; \mathbb P\text{-a.s.},\; \forall \p\in \mathcal P_0.
\end{align}
We define the bi-conjugate of $H^b$ as follows
\[ \hat{H}^{b}(t,x,y,z,u,\gamma):=\sup_{a\in \mathcal S_d^{\geq 0}} \Big\{\frac12 Tr(a\gamma) - F^b_t(x,y,z,u,a)\Big\}.\]

We introduce the following fully nonlinear PDEs
\[(PDE)\begin{cases}
\partial_t v(t,x)+\hat{H}^{b}(t,x,v(t,x),\nabla v(t,x),g(t,x)-v(t,x),\Delta v(t,x))=0,\; t<T\\
v(T,x)=g(T,x),\; x\in \mathbb R^d.
\end{cases}\]

\begin{lemma}[Theorem 5.3 in \cite{soner2012wellposedness}]\label{lemmaFK}
Let $v\in \mathcal C^{1,2}([0,T]\times \mathbb R^d)$ be a solution to (PDE). Then the auxiliary 2BSDE \eqref{2BSDEb:eq:markov} admits a unique solution given by
\[Y^b_t:=v(t,X_t),\; Z^b_t:=\nabla v(t,x),\; K^b_t:=\int_0^t k^b_s ds,\]
where 
\[k^b_s=\hat{H}^{b}(s,X_s,Y^b_s,Z^b_s,g(s,X_s)-Y^b_s,\Gamma_s),\; \Gamma_s=\Delta v(s,X_s).\]
\end{lemma}
\noindent We now turn to the PDE characterization of the value function as the equivalent of Theorem \ref{thm:decompbsde} in the Markovian setting. 

\begin{theorem}[Piecewise Feynman-Kac representation formula]
    The 2BSDE \eqref{2BSDEmarkov} admits a unique solution given by 
\[Y_t= v(t,X_t) \mathbf 1_{t<\tau} + g(\tau,X_\tau) \mathbf 1_{t\geq \tau},\, Z_t=\nabla v(t,X_t) \mathbf 1_{t<\tau},\]
\[K_t= K_t^{b}  \mathbf 1_{t<\tau},\; U_t=(g(t,X_t)-v(t,X_t))\mathbf 1_{t<\tau},\; t<T\]
where $K_t^b$ is given by Lemma \ref{lemmaFK} and $v$ is a classical solution to (PDE).
    
\end{theorem}

\begin{proof}
  This is a direct consequence of Lemma \ref{lemmaFK} together with Theorem \ref{thm:decompbsde}.
\end{proof}
\begin{remark}
    Note that (PDE) is not an integropartial PDE and the jumps and the erratic time $\tau$ do not play any role in the PDE itself to describe the solution to \eqref{2BSDEmarkov}. In particular, $Y_t,Z_t,U_t,K_t$ are not Markovian functions of $t,X_t$ since $\lambda$ may be a non-Markovian process. However, we rely on the Markovianity of the auxiliary 2BSDE \eqref{2BSDEb:eq:markov} to connect the solution to \eqref{2BSDEmarkov} with the fully nonlinear PDE (PDE). 
\end{remark}

\section{Application to stochastic control in erratic environments}
\label{sec:erraticcontrol}
This section applies the preceding results to stochastic control problems under weak formulations. We examine two types of controlled diffusion optimization: one where both the drift and volatility are controlled by the agent, and another where only the drift is controlled in the presence of volatility uncertainty. The concept of controlled volatility and weak formulation has been well understood and extensively studied in the literature, particularly following the influential work of \cite{cvitanic2018dynamic,zhang2017stochastic}, while the notion of volatility ambiguity has been investigated, for example, in \cite{bordigoni2007stochastic,nutz2012superhedging,bielecki2019adaptive}. This last case is particularly relevant for real-world applications, where modelers inevitably encounter uncertainty in parameter estimation. In this section, we extend this class of problems by adding uncertainty to the termination of the optimization problems studied. 
\subsection{A controlled model to unify volatility controlled and ambiguity}
\label{sec:controlmodel}
We define $\mathcal{A}$ and $\mathcal{B}$ as the sets of $\mathbb{F}$-adapted processes taking values in two subsets $A\subset\mathbb R^m$ and $B\subset [\varepsilon,\overline\beta]^m$, for $0<\varepsilon<\overline\beta$, for $m\geq 1$. We call \textit{control process} every pair $(\alpha, \beta) \in \mathcal{A} \times \mathcal{B}$. We will also denote, to make the notation lighter,  $\nu = (\alpha, \beta)$ and its respective set $\mathcal{V}$. To clarify the notations, for the rest of the paper:
\begin{itemize}
    \item in the volatility control setting, $\alpha$ and $\beta$ will be directly controlled by an agent, where $\beta$ will affect only the volatility part and $\alpha$ the drift;
    \item in the volatility ambiguity model, $\beta$ has the interpretation of a worst-case scenario, chosen by an exogenous actor, named the Nature, picking the worst volatility coefficient for the optimization problem.
\end{itemize}

We specify the drift and volatility coefficient as follows:
\[
\mu : [0,T] \times \Omega \times A \times B \longrightarrow \mathbb{R}^n, \quad \text{ such that $\mu(\cdot,a, b)$ is $ \ \mathbb{F}\text{-optional for any } (a, b) \in A\times B$ and bounded},
\]
\[
\sigma : [0,T] \times \Omega \times B \longrightarrow \mathcal{M}_{n \times d}(\mathbb{R}), \quad \text{ such that $\sigma(\cdot, b)$ is $ \ \mathbb{F}\text{-optional for any } b \in B$ and bounded}.
\]

We now recall the output process definition and the control problem. In order to make the construction suitable for the volatility ambiguity, differently from the seminal work of \cite{cvitanic2018dynamic} and more in line with the work of \cite{hernandez2019contract}, we start by defining the driftless SDE driven by the \( d \)-dimensional Brownian motion \( W \)
\begin{equation}
    X^{t,x,\beta}_s = x(t) + \int_t^s \sigma(r, X^{t,x,\beta}_{\cdot\wedge r}, \beta_r) dW_r, \quad s \in [t, T] \label{eq:Driftless SDE}
\end{equation}
\[
X^{t,x,\beta}_r = x(r),\; x\in \Omega \quad r \in [0, t],
\]
where $X_{\cdot\wedge t}$ denotes the path of $X$ up to time $t$. We now define a unified control model for both volatility controlled by an agent or ambiguity controlled by an exogenous player called the Nature as a set of weak solutions as follows. 
\begin{definition}
   We say that $(\mathbb{P}, \beta)$ is a weak solution for \eqref{eq:Driftless SDE} if the distribution of $X^{t,x, \beta}_t$ under $ \mathbb{P}$ is $\delta_{x(t)}$, where $\delta$ denotes the Dirac measure and there exists a $\mathbb P$-Brownian motion, denoted by $W^{\mathbb{P}}$, such that
\[
X_s = x(t) + \int_t^s \sigma(r, X_{\cdot\wedge r}, \beta_r) dW_r^{\mathbb{P}}, \quad s \in [t, T], \quad \mathbb{P}\text{-a.s.}
\]
We denote by
\[
    \mathcal{N}(t, x) := \left\{ (\mathbb{P}, \beta): \text{ weak solutions of \eqref{eq:Driftless SDE}}\right\}
\]
and by $\mathcal{P}(t, x)$ the set of probability measures such that there exists an $\F$-optional process $\beta$ such that the probability coupled with the process $\beta$ is a weak solution of \eqref{eq:Driftless SDE}.
\end{definition}

\begin{remark}
    Note that $\left\{ \mathcal{P}(t, x), \, t \in [0, T] \times \Omega \right \}$ is a saturated family of probability measures, see for example \cite{hernandez2019contract, cvitanic2018dynamic} and the references therein.
\end{remark} 

We now turn to the drift-controlled model. From now on, we denote by $(\mathcal{E}(Z_t))_{t \in [0,T]}$ the Doleans-Dade exponential of a continuous stochastic process $Z$ defined by 
\[\mathcal E(Z_t):= e^{Z_t-\frac12[Z,Z]_t}.\]

\begin{definition}[Admissible agent's effort]
    A control process $\alpha$ is said to be admissible, if for every $(\mathbb{P}, \beta) \in \mathcal{N}(t, x)$ the following process is an $(\mathbb{F}, \mathbb{P})$--martingale
    \[\mathcal E:=
    \left( \mathcal{E} \left( \int_0^t \sigma^\top(\sigma \sigma^\top)^{-1}(s, X_{\cdot\wedge s}, \beta_s) \mu(s, X_{\cdot\wedge s}, \alpha_s, \beta_s) \cdot dW^{\mathbb{P}}_s \right) \right)_{t \in [0,T]}.
    \]
    We denote by $\mathcal A$ the set of such admissible efforts.
\end{definition}
\begin{remark}
    Note that this set is not empty if we assume that $\sigma$ is such that $\sigma(\sigma\sigma^\top)^{-1}$ is bounded, for example if $n=d=1$ and $\sigma:[0,T]\times\mathbb R\times B\longrightarrow [\underline\sigma,\overline\sigma]$ with $0<\underline\sigma<\overline\sigma$. Alternatively, if we assume that $\mu$ can be decoupled multiplicatively as follows $\mu(s,x,\new{\alpha},\beta)=\lambda(s,x,\new{\alpha})\sigma(s,x,\beta)$ for some bounded function $\lambda$, the stochastic exponential in the definition of admissible effort becomes
    \[  
    \left( \mathcal{E} \left( \int_0^t \lambda(s, X_{\cdot\wedge s}, \alpha_s, \beta_s) \cdot dW^{\mathbb{P}}_s \right) \right)_{t \in [0,T]},
    \] and is directly a martingale, similarly to \cite{cvitanic2018dynamic}.
\end{remark}

Consider an admissible effort $\alpha \in \mathcal{A}$ and $(t, x) \in [0, T] \times \Omega$. For every subset $\mathcal{N} \subset \mathcal{N}(t, x)$ define
\[
\mathcal{N}^{\alpha} := \left\{ (\mathbb{P}^{\alpha}, \beta), \ \frac{d\mathbb{P}^{\alpha}}{d\mathbb{P}} = \mathcal{E} \left( \int_{t\wedge \tau}^{\T} \sigma^\top (\sigma \sigma^\top)^{-1}(s, X_{\cdot\wedge s}, \nu_s) \mu(s, X_{\cdot\wedge s}, \alpha_s, \beta_s) \cdot dW_s^{\mathbb{P}} \right), \ (\mathbb{P}, \beta) \in \mathcal{N} \right\}.
\]

By Girsanov’s Theorem we then deduce that for any $\alpha \in \mathcal{A}$, and for any $(\mathbb{P}^{\alpha}, \beta) \in \mathcal{N}^{\alpha}$, 
\[
X_s = x(t) + \int_t^s \mu(r, X_{\cdot\wedge r}, \alpha_r, \beta_r) dr + \int_t^s \sigma(r, X_{\cdot\wedge r}, \beta_r) dW_r^{\mathbb{P}^{\alpha}}, \quad s \in [t, T], \quad \mathbb{P}^{\alpha}\text{--a.s.}
\]

We point out that this construction enables us to unify the mathematical frameworks for both the controlled volatility \cite{cvitanic2018dynamic} and the ambiguity \cite{hernandez2019contract} cases.  Especially for the volatility ambiguity, the Nature fixes a diffusion coefficient first through a change of probability, selecting a family of probability measures. Then, within this family, the agent can choose its optimal effort which is changing the probability measure through a Girsanov transformation. In the case of volatility control, this is still a valid construction, identifying the Nature with the agent, equivalent to the ones proposed in \cite{cvitanic2018dynamic, denis2024second}.\\

\noindent We now turn to the optimization problems for both cases: volatility controlled and $sup\; sup$ optimization or ambiguity and $sup\; inf$ problem. We first define a discount factor by
\[
    \mathcal K_t^{\alpha, \beta} := \exp\left(- \int_0^t k(s,X_{\cdot\wedge s},\alpha_s, \beta_s) ds \right), \quad t \in [0,T],
\]
where $k$ is a bounded and $\F$-optional function. For any control $(\alpha;(\mathbb P^\alpha,\beta))\in \mathcal A\times \mathcal N^\alpha$, the objective function of the agent is thus given by
\[
     J(\alpha;\mathbb{P}^{\alpha}, \beta) =\E^{\mathbb{P}^{\alpha}}[ \mathcal K_{\T}^{\alpha, \beta}\Phi(\xi) - \int_0^{\T}  \mathcal K_t^{\alpha, \beta}c_t(\alpha_t,\beta_t) dt]
\]
where 
\begin{itemize}
\item $\xi$ is the liability the agent is facing as an $\mathcal G_{\T}-$measurable random variable, $\Phi$ is a utility function assumed to be continuous and non-decreasing such that $\|\Phi(\xi)\|_{L_{0,\hat\chi}^{2}}<\infty$.
\item $c:[0,T]\times \Omega\times A\times B\longrightarrow \mathbb R$ is a cost function possibly depending on the time, and the path of the process $X$ such that 

\[\sup_{\alpha\in \mathcal A}\, \sup_{(\mathbb P^\alpha,\beta)\in \mathcal N^\alpha} \mathbb E^{\mathbb P^\alpha}\Big[\int_0^{T} e^{-\hat\chi \Lambda_s}|c_s(\alpha_s,\beta_s)|^2\,ds\Big]<\infty.\]
\end{itemize}
Given this objective function, the two problems differ in their optimization routines.\\

\noindent \textbf{Agent as a volatility controller.} In the purely controlled case, we aim to maximize the expected utility over the two controls, volatility and drift simultaneously, that is to solve
    \begin{equation}\label{value+}
       \overline{V_0}:=\sup_{\alpha\in \mathcal A}\; \sup_{(\mathbb{P}^{\alpha}, \beta) \in \mathcal{N}^{\alpha}} J(\alpha;\mathbb{P}^{\alpha}, \beta)
    \end{equation}
\noindent \textbf{Nature as the volatility controller.} In a volatility ambiguity setting, we seek to optimize the worst-case scenario, tackling a problem of the form
  \begin{equation}\label{value-}
       \underline{V_0}:=\sup_{\alpha\in \mathcal A}\; \inf_{(\mathbb{P}^{\alpha}, \beta) \in \mathcal{N}^{\alpha}} J(\alpha;\mathbb{P}^{\alpha}, \beta)
\end{equation}
    We want to stress that this problem can be seen as the Nature being an adversarial opponent in a stochastic game, where the Nature picks its "strategy" $\beta$ to make our output process as bad as possible and we choose our strategy $\alpha$ to mitigate the output. This is a worst-case scenario or robust optimization framework.\\
    
\noindent We aim to prove that both problems can be linked to the solution of 2BSDE \eqref{2bsde:eq}. The idea is that, in both cases, we can represent the values of both problems $\overline{V_0},\underline{V_0}$ as the solution to a 2BSDE we studied in Section \ref{sec:2BSDE}.

\subsection{Supporting 2BSDEs and driver definitions}\label{sec:control2bsde}
In this section, we introduce the 2BSDEs related to the optimization problems above. The interpretation of the two cases is different, and even if the 2BSDEs look similar, with the exception of their driver, the volatility ambiguity case will require the extra assumption of the Isaacs condition from the seminal work in stochastic differential games of \cite{isaacs1999differential}. \\

We start by defining the driver $f$ for our 2BSDEs:
\[ 
    f : [0, T] \times \Omega \times \mathbb{R} \times \mathbb{R}^d \times \mathbb{R} \times A \times B \times \mathbb{R}^+ \rightarrow \mathbb{R} 
\]
defined by
\[
f(t, x, y, z, u, a, \beta, \lambda) := -k(t, x, b, b) y - c_t(x, a, b) + \mu(t, x, a, b) \cdot z + \lambda u.
\]

Define also for every \( (t, x, \Sigma) \in [0, T] \times \Omega \times \mathbb{S}_d^+ \) the set
\[
V_t(x, \Sigma) := \left\{ b \in B,\, \sigma(t, x, b)\sigma^\top(t, x, b) = \Sigma \right\},
\]
and denote by \( \mathbf{B}(\hat{\sigma}^2) \) the set of controls \( \beta \in \mathcal{B} \) with values in \( V_t(x, \hat{\sigma}_t^2) \), \( dt \otimes \mathbb{P} \)-a.e., where 
\[\hat\sigma_t:=\lim_{\varepsilon\downarrow 0} \frac{\langle X\rangle_t-\langle X\rangle_{t-\varepsilon}}{\varepsilon},\; t\in [0,T],\]
taking values in $\mathcal S_d^+$ and where $\langle X\rangle$ coincides with the quadratic variation of $X$ for every probability measure $\mathbb P\in \mathcal P(t,x)$ by \cite{karandikar1995pathwise}.\\

From this generator definition, we define two Hamiltonians, one associated with the volatility-controlled problem, and the other one associated with the volatility ambiguity problem.
\begin{itemize}
    \item \textbf{Agent as a volatility controller.} The Hamiltonian \( \overline F : [0, T] \times \Omega \times \mathbb{R} \times \mathbb{R}^d \times \mathbb{R} \times \mathcal{S}^+_d \rightarrow \mathbb{R} \) associated with the volatility controlled problem is defined by
        \[
            \overline{F}(t, x, y, z, u, \lambda,\Sigma) := \sup_{b \in V_t(x, \Sigma)} \sup_{a \in A} f(t, x, y, z, u, a, b, \lambda)
        \]
    \item \textbf{Nature as an adversarial volatility controller.} The Hamiltonian \( \underline F : [0, T] \times \Omega  \times \mathbb{R} \times \mathbb{R}^d \times \mathbb{R} \times \mathcal{S}^+_d \rightarrow \mathbb{R} \) associated with the volatility ambiguity problem is defined by
        \[
            \underline F(t, x, y, z, u, \lambda, \Sigma) := \inf_{b \in V_t(x, \Sigma)} \sup_{a \in A} f(t, x, y, z, u, a, b, \lambda).
        \]
\end{itemize}

Given the drivers, we can define the 2BSDEs that we will use equivalently to solve our control problems.
Consider the following 2BSDEs
\begin{align} \label{eq:2bsdecontr:+}
    \overline Y_t &= \Phi(\xi) + \int_{t\wedge \tau}^{\T} \overline F(s, X_s, \overline Y_s, \overline Z_s,\overline U_s,\lambda_s, \hat{\sigma}^2_s) ds - \int_{t\wedge \tau}^{\T} \overline Z_s \cdot dX_s  \\ \nonumber
    & + \int_{t\wedge \tau}^{\T} d\overline K_s -\int_{t\wedge \tau}^{\T}\overline U_s dH_s,\,\text{ for all }0\leq t\leq T,\; \mathbb P\text{-a.s.},\; \forall \mathbb P\in \mathcal P(t,x), 
\end{align}
and
\begin{align} \label{eq:2bsdecontr:-}
    \underline Y_t &= \Phi(\xi) + \int_{t\wedge \tau}^{\T} \underline F(s, X_s, \underline Y_s, \underline Z_s,\underline U_s,\lambda_s, \hat{\sigma}^2_s) ds - \int_{t\wedge \tau}^{\T} \underline Z_s \cdot dX_s  \\ \nonumber
    & - \int_{t\wedge \tau}^{\T} d\underline K_s -\int_{t\wedge \tau}^{\T}\underline U_s dH_s,\,\text{ for all }0\leq t\leq T,\; \mathbb P\text{-a.s.},\; \forall \mathbb P\in \mathcal P(t,x), 
\end{align}

\begin{remark}
  Note that we could define the 2BSDE with an orthogonal martingale $M$. However, the proof of our main results below would be then taken as a conditional expectation under a saturated probability measure under which $M$ is a martingale and therefore this part vanishes. For ease of notation and exposition, we will therefore ignore that term. 
\end{remark}
\begin{remark}
    Given the assumption made about the intensity of the default $\lambda$, the proof below is similar to \cite{cvitanic2018dynamic, hernandez2019contract, possamai2018stochastic, denis2024second}. This is because the aggregation of the process $K$ holds thanks to the aggregation of the process $U$, which is only possible when the intensity is independent of the probability measure, by using \cite{nutz2012superhedging}.  
\end{remark}

Since $k,\sigma,\mu$ are bounded and since $c$ does not depend on $y,z,u$, we have the following Lemma ensuring that the 2BSDEs \eqref{eq:2bsdecontr:+} and \eqref{eq:2bsdecontr:-} are well-posed in the sense of Assumption \ref{assumption:L}.
\begin{lemma}\label{lemmaL}
    For any 
\((t, x, y, z, u, \lambda, \Sigma) \in [0, T] \times \Omega \times \mathbb{R} \times \mathbb{R}^d \times \mathbb{R} \times \mathbb{R}^+ \times \mathcal S_d^+\), 
the maps $\overline F$ and $\underline F$ satisfy Assumption \ref{assumption:L}.
\end{lemma}

\subsection{The volatility control case: from full control to 2BSDE}
\label{sec:control}

In this section, we turn to the solution of the non-Markovian optimization \eqref{value+} and its connection with the 2BSDE \eqref{eq:2bsdecontr:+}.

\begin{theorem}
Let \((\overline Y, \overline Z, \overline U, \overline K)\) be the unique solution to \eqref{eq:2bsdecontr:+}. Then, the value function of the agent \eqref{value+} is given by
\[
\overline{V_0}(\xi) = \sup_{\alpha \in \mathcal{A}} \sup_{(\mathbb{P}, \beta) \in \mathcal{N}^\alpha} \mathbb{E}^\mathbb{P}[\overline Y_0].
\]

Moreover, \((\alpha^\star, \mathbb{P}^\star, \beta^\star) \) are optimal if and only if \((\alpha^\star, \mathbb{P}^\star, \beta^\star) \in \mathcal{A} \times \mathcal{N}^{\alpha^\star}\) such that:
\begin{itemize}
    \item[(i)] \((\alpha^\star, \beta^\star)\) attains the supremum in the definition of \(\overline F(\cdot, X, \overline Y, \overline Z, \overline U,\lambda,\hat\sigma^2) \), \(dt \otimes \mathbb{P}^\star\)-a.e.,
    \item[(ii)] \(\overline K_{\T} = 0\), \(\mathbb{P}^\star\)--a.s.
\end{itemize}
\end{theorem}

The existence of a unique solution to the 2BSDE \eqref{eq:2bsdecontr:+} is a direct consequence of the assumptions made on $k,\sigma,\mu,c$ and Theorem \ref{thm:decompbsde} and Theorem \ref{thm:lipschitz}. We omit the proof regarding the optimizers of the problem since it follows the same lines as the one of the ambiguity case below and extend \cite{cvitanic2018dynamic} to the erratic horizon $\tau\wedge T$. We refer to either an extension to \cite[Proposition 5.4]{cvitanic2018dynamic} or the proof of Theorem \ref{thm:ambiguity} below. 

\subsection{The volatility ambiguity case: from adversarial nature to 2BSDE}
\label{Section:ambiguity}

We now turn to the case where the agent is facing an adversarial player, the Nature, regarding the volatility control. Similarly to the previous section, we aim at connecting the robust optimization \eqref{value-} with the 2BSDE \eqref{eq:2bsdecontr:-}. Note that this is a worst-case scenario framework and its link to 2BSDE has been previously studied in \cite{hernandez2019contract,mastrolia2025agency}. We extend this result to the erratic horizon $T\wedge \tau$. In contrast to the previous case, where boundedness alone ensured well-posedness, the ambiguity setting introduces the additional difficulty of a potentially unbounded sup inf problem, where the second sub-problem becomes $-\infty$ once the worst strategy of the Nature is considered. To overcome this technical difficulty, we will add another assumption ensuring the existence of an optimizer, known as Isaacs condition, from the seminal work of \cite{isaacs1999differential}. This assumption has been also made in the literature regarding zero-sum games in weak formulation, see for example \cite{possamai2020zero,hernandez2019contract}. We assume that the following condition is satisfied.
\begin{assumption}\label{assumption:isaac}
For any $(t, x, y, z, u, \lambda, \Sigma) \in [0, T] \times \Omega \times \mathbb{R} \times \mathbb{R}^d \times \mathbb{R} \times \mathbb{R}^+ \times \mathcal{S}^+_d$

\[
\inf_{b \in V_t(x, \Sigma)} \sup_{a \in A} f(t, x, y, z, u, a,b, \lambda)
=
\sup_{a \in A} \inf_{b \in V_t(x, \Sigma)} f(t, x, y, z, u, a, b, \lambda).
\]    
\end{assumption}

    We refer to \cite[Section 4]{hernandez2019contract} for explicit examples satisfying this condition.

\begin{theorem}\label{thm:ambiguity}
Let Assumption \ref{assumption:isaac} be true. Let \((\underline Y, \underline Z, \underline U, \underline K)\) be the unique solution to \eqref{eq:2bsdecontr:-}. Then, the value function of the agent \eqref{value-} is given by
\[
\underline{V_0}(\xi) = \sup_{\alpha \in \mathcal{A}} \inf_{(\mathbb{P}, \beta) \in \mathcal{N}^\alpha} \mathbb{E}^\mathbb{P}[\underline Y_0].
\]

Moreover, \((\alpha^\star, \mathbb{P}^\star, \beta^\star) \) are optimal if and only if \((\alpha^\star, \mathbb{P}^\star, \beta^\star) \in \mathcal{A} \times \mathcal{N}^{\alpha^\star}\) such that:
\begin{itemize}
    \item[(i)] \((\alpha^\star, \beta^\star)\) attains the saddle point sup inf in the definition of \(\underline F(\cdot, X, \underline Y, \underline Z, \underline U,\lambda,\hat\sigma^2) \), \(dt \otimes \mathbb{P}^\star\)--a.e.,
    \item[(ii)] \(\underline K_{\T} = 0\), \(\mathbb{P}^\star\)--a.s.
\end{itemize}
\end{theorem}

\begin{proof}

From Lemma \ref{lemmaL} together with Theorem~\ref{thm:decompbsde} and Theorem \ref{thm:lipschitz} we deduce that there exists a unique solution \((\underline Y, \underline Z, \underline U, \underline K)\)  to \eqref{eq:2bsdecontr:-}. We define for each $\alpha\in \mathcal A$ the following 2BSDE with erratic horizon

    \begin{align}\label{eq:2bsdeambinf}
        Y^\alpha_t = \Phi(\xi) &+ \int_{t\wedge \tau}^{\T} \inf_{\beta\in V_t(X_{\cdot\wedge t},\hat\sigma_t^2)}f(s, X_s, Y^\alpha_s, Z^\alpha_s,U_s^\alpha,\alpha_s, \beta_s,\lambda_s) ds - \int_{t\wedge \tau}^{\T} Z^\alpha_s \cdot dX_s  \nonumber \\ 
        & - \int_{t\wedge \tau}^{\T} dK^\alpha_s -\int_{t\wedge \tau}^{\T}U^\alpha_s dH_s,\; \mathbb P\text{-a.s.},\; \mathbb P\in \mathcal P(t,x).  
    \end{align}
 Similarly, we consider the 2BSDE when neither $\alpha$ nor $\beta$ are controlled:

     \begin{align}\label{eq:2bsdeabonly}
        Y^{\alpha,\beta}_t = \Phi(\xi) &+ \int_{t\wedge \tau}^{\T} f(s, X_s, Y^{\alpha,\beta}_s, Z^{\alpha,\beta}_s,U_s^{\alpha,\beta},\alpha_s, \beta_s,\lambda_s) ds - \int_{t\wedge \tau}^{\T} Z^{\alpha,\beta}_s \cdot dX_s  \nonumber \\ 
        & - \int_{t\wedge \tau}^{\T} dK^{\alpha,\beta}_s -\int_{t\wedge \tau}^{\T}U^{\alpha,\beta}_s dH_s,\; \mathbb P\text{-a.s.},\; \mathbb P\in \mathcal P(t,x).  
    \end{align}

We finally denote by $( \mathcal{Y}^{\alpha, \beta}, \mathcal{Z}^{\alpha, \beta}, \mathcal{U}^{\alpha, \beta})$ the unique solutions to the Lipschitz BSDE under some probability $\mathbb P'\in \mathcal P_0$ 
    \begin{align}
\nonumber
        \mathcal{Y}^{\mathbb P';\alpha, \beta}_t& = \Phi(\xi) + \int_{t\wedge \tau}^{\T} f(s, X_s, \mathcal{Y}^{\mathbb P';\alpha, \beta}_s, \mathcal{Z}^{\mathbb P';\alpha, \beta}_s, U_s^{\mathbb P';\alpha, \beta},\alpha_s, \beta_s,\lambda_s) ds\\
            \label{bsde:ab} &- \int_{t\wedge \tau}^{\T} \mathcal{Z}^{\mathbb P';\alpha, \beta}_s \cdot dX_s  -\int_{t\wedge \tau}^{\T} \mathcal{U}^{\mathbb P';\alpha, \beta}_s dH_s.
    \end{align}
From the dynamic programming principle Proposition \ref{DPP} together with Remark \ref{rkDPP}, we get
    \[
        Y^{\alpha}_0 = \text{ess inf}^{\mathbb P}_{\mathbb{P}' \in \mathcal{P}(0,\mathbb{P}, \mathbb{G})} \mathcal{Y}_0^{\mathbb{P}', \alpha}, \quad \mathbb{P}\text{-a.s. for all } \mathbb{P} \in \mathcal{P}_0.
    \]

    and
    \[
        Y^{\alpha, \beta}_0 = \text{ess inf}^{\mathbb P}_{\mathbb{P}' \in \mathcal{P}(0,\mathbb{P}, \mathbb{G})} \mathcal{Y}_0^{\mathbb{P}', \alpha, \beta}, \quad \mathbb{P}\text{-a.s. for all } \mathbb{P} \in \mathcal{P}_0.
    \]

  From the comparison theorem \ref{comparaisontheorem}, we then deduce that for any $\p\in \mathcal P_0$
    \begin{align*}
        \underline Y_0 
        &= \text{ess sup}^{\mathbb P}_{\alpha \in \mathcal{A}} Y_0^{\alpha} =  \text{ess sup}^{\mathbb P}_{\alpha \in \mathcal{A}} \text{ess inf}^{\mathbb P}_{\beta \in {V_t(X_{\cdot\wedge t},\hat{\sigma}_t^2)}} Y_0^{\alpha, \beta}\\
        &= \text{ess sup}^{\mathbb P}_{\alpha \in \mathcal{A}} \text{ess inf}^{\mathbb P}_{\beta \in {V_t(X_{\cdot\wedge t},\hat{\sigma}_t^2)}} \text{ess inf}_{\mathbb{P}' \in \mathcal{P}(0,\mathbb P,\mathbb{G})}^{\mathbb{P}} \mathcal{Y}_0^{\mathbb{P'}, \alpha, \beta}\\
        &=\text{ess sup}^{\mathbb P}_{\alpha \in \mathcal{A}} \text{ess inf}^{\mathbb P}_{\beta \in {V_t(X_{\cdot\wedge t},\hat{\sigma}_t^2)}} \text{ess inf}_{\mathbb{P}' \in \mathcal{P}(0,\mathbb P,\mathbb{G})}^{\mathbb{P}} \mathbb E^{\mathbb P'_\alpha}[ \mathcal K_{\T}^{\alpha, \beta}\Phi(\xi) - \int_0^{\T}  \mathcal K_t^{\alpha, \beta}c_t(\alpha_t,\beta_t) dt],
    \end{align*}
where we have used a linearization technique for the last equality with $\mathbb P'_\alpha$ defined by 
 \[
        \frac{d\mathbb{P}'_{\alpha}}{d\mathbb{P}'} = \mathcal{E} \left( \int_{.}^T \sigma^\top (\sigma \sigma^\top)^{-1}(s, X, \beta_s) \mu(s, X, \alpha_s, \nu_s) \cdot dW_s^{\mathbb{P}} + \int_{.}^T U_s dH_s\right).
    \]
    Consequently and taking a supremum on $(\alpha;\mathbb P,\beta)\in \mathcal A\times\mathcal N^\alpha$ we get

    \[
\underline{V_0}(\xi) = \sup_{\alpha \in \mathcal{A}} \inf_{(\mathbb{P}, \beta) \in \mathcal{N}^\alpha} \mathbb{E}^\mathbb{P}[\underline Y_0].
\] Finally, going back to the previous steps,  $(\alpha^\star,\mathbb P^\star,\beta^\star)$ is optimal if and only if all the essential suprema and infima above are attained, that is if and only if $(\alpha^\star,\beta^\star)$ is a saddle point and satisfies Isaacs condition, and $\mathbb P^\star$ is such that 
\[\mathbb E^{\mathbb P^\star}[\underline Y_0]= \mathcal Y^{\mathbb P^\star,\alpha^\star,\beta^\star}_0,\] that is $\underline K_{\T} = 0,\, \mathbb{P}^\star-a.s.$
\end{proof}

\newpage

\bibliographystyle{plain}
\small
\bibliography{biblio}

\newpage
\appendix

\section{Quadratic 2BSDE for CARA utility maximization under
ambiguity and erratic horizon}

\label{appendix:cara}

We propose an extension of this study to CARA utility function beyond the separable case studied above.  For the sake of consistency with control in erratic environment terminology we focus on the ambiguity case. Throughout this
subsection we work under the standing assumptions of
Section~\ref{sec:controlmodel} and we add the following bounded-data
assumption. We enforces the boundedness of the terminal condition $\xi$ and the cost function $c$ with the following assumption.

\begin{assumption}\label{assumption:CARA-bounded}
$\xi$ is bounded $\mathbb P\text{-a.s.}$ for every $\mathbb P\in\mathcal P_0$; the running cost $c$ is bounded
on $[0,T]\times\Omega\times A\times B$ with $\|c\|_\infty\le \overline c$;
the discount $k$, the drift $\mu$, and the volatility $\sigma$ are
bounded as in Section~\ref{sec:controlmodel}; the
$\mathcal P_0$-density Hypothesis~\ref{densityhyp} and the Isaacs
condition (Assumption~\ref{assumption:isaac}) hold; and the
$\hat p$-integrability of $\lambda$ from Section~\ref{sec:math}
holds. 
\end{assumption}

For a risk aversion parameter $\eta>0$ and for an admissible
effort $\alpha\in\mathcal A$, define the conditional CARA objective function
\[
J_t(\alpha;\mathbb P^\alpha,\beta)
:= -\mathbb E^{\mathbb P^\alpha}\!\left[\exp\!\left(-\eta\Big(\xi-\int_{t\wedge\tau}^{T\wedge\tau}c_s(\alpha_s,\beta_s)\,ds\Big)\right)\bigg|\mathcal G_t\right].
\]

In this section, we are interested to solve the following worst-case optimal control with erratic horizon:

\begin{equation}\label{eq:CARAOpt}
V_0^A(\xi)
:= \mathrm{sup}_{\alpha\in\mathcal A}\,
\mathrm{ess\,inf}_{(\mathbb P^\alpha,\beta)\in\mathcal N^\alpha}
J_0(\alpha;\mathbb P^\alpha,\beta),
\end{equation}
together with the dynamic version of the value function
\[
V_t^A(\xi)
:= \mathrm{ess\,sup}_{\alpha\in\mathcal A}\,
\mathrm{ess\,inf}_{(\mathbb P^\alpha,\beta)\in\mathcal N^\alpha}
J_t(\alpha;\mathbb P^\alpha,\beta),\qquad 0\le t\le T.
\]
We define
\[
\;Y_t := -\tfrac{1}{\eta}\log\!\left(-V_t^A(\xi)\right),\qquad 0\le t\le T\wedge\tau.
\]
\begin{remark}\label{rem:bounded}
Note that the boundedness of $\xi$ and~$c$ in
Assumption~\ref{assumption:CARA-bounded} guarantees
\[
-\,e^{\eta(\|\xi\|_\infty+T\|c\|_\infty)}\;\le\;V_t^A(\xi)\;\le\;-\,e^{-\eta(\|\xi\|_\infty+T\|c\|_\infty)}\;<\;0,
\]
so $Y$ is well-defined and bounded.
\end{remark}

Define, for $(t,x,y,z,u,\lambda,\Sigma)\in[0,T]\times\Omega\times\mathbb R\times\mathbb R^d\times\mathbb R\times\mathbb R^+\times\mathcal S_d^+$,
\begin{equation}\label{eq:CARA-Hamiltonian}
\underline F^{\eta}(t,x,y,z,u,\lambda,\Sigma)
:= \inf_{\beta\in V_t(x,\Sigma)}\sup_{\alpha\in A}
\Big\{\,
-c_t(\alpha,\beta) + \mu(t,x,\alpha,\beta)\cdot z
+ \tfrac{\eta}{2}\,\|\Sigma^{1/2} z\|^2
+ \lambda\,\tfrac{e^{-\eta u}-1}{\eta}
\,\Big\}.
\end{equation}
\begin{remark}
    Unlike the classical framework of \cite{rouge2000pricing, hu2005utility}, the extension to the worst-case approach combined with the erratic horizon leads to an inf-sup optimization with the intensity of the default given by the $\lambda$ term. 
\end{remark}

We introduce the following 2BSDE with erratic horizon, whose
generator is the CARA Hamiltonian \eqref{eq:CARA-Hamiltonian} previously defined:
\begin{align}\label{eq:CARA-2bsde}
Y_t \;=\;\xi
&-\int_{t\wedge\tau}^{T\wedge\tau}\underline F^\eta(s,X,Y_s,Z_s,U_s,\lambda_s,\hat\sigma_s^2)\,ds
-\int_{t\wedge\tau}^{T\wedge\tau} Z_s\cdot dX_s^{c,\mathbb P}
-\int_{t\wedge\tau}^{T\wedge\tau} dM_s^{\mathbb P}\\[-2pt]
\nonumber
&+(K_{T\wedge\tau}^{\mathbb P}-K_{t\wedge\tau}^{\mathbb P})
-\int_{t\wedge\tau}^{T\wedge\tau} U_s\,dH_s,
\qquad \forall\,t\in[0,T],\;\mathbb P\text{-a.s.},\;\forall\,\mathbb P\in\mathcal P(t,x).
\end{align}

\begin{proposition}
\label{prop:CARA-existence}
Under Assumption~\ref{assumption:CARA-bounded},
2BSDE~\eqref{eq:CARA-2bsde} admits a unique solution
$(Y^\eta,Z^\eta,U^\eta,(K^{\eta,\mathbb P}),(M^{\eta,\mathbb P}))$
in the sense of Definition~\ref{def:2bsde} with bounded $Y^\eta$.
\end{proposition}

\begin{proof}
We use the theory of BMO martingale with quadratic BSDE, see for example
\cite{briand2008quadratic,morlais2009quadratic,kazi2015second} to the
present erratic-horizon setting in four steps.  First, we introduce the truncated generator corresponding 2BSDE with solution $(Y^n,Z^n,U^n,K^n)$ recovering the Lipschitz setting studied above. Then, we show that the solutions $Y^n,U^n$ are uniformly bounded in $n$ together with the BMO-norm of $Z^n$. Finally, we take the limit showing the existence of the 2BSDE together with the uniqueness. for the rest of this proof, we set
$\overline R := \|\xi\|_\infty+T\|c\|_\infty$.

\emph{Step 1.}  For each $n\ge 1$ define the
truncation operator
\[T_n(z,u):=(\,n\wedge\|z\|\cdot z/\|z\|,\; (-n)\vee(u\wedge n)\,)\]
componentwise on $z$, with the convention $z/\|z\|:=0$ at $z=0$,
and let
\[\underline F^{\eta,n}(t,x,y,z,u,\lambda,\Sigma):=\underline F^\eta(t,x,y,T_n(z,u),\lambda,\Sigma).\]
For each $\Sigma\in\mathcal S_d^+$ the map
$(y,z,u)\mapsto \underline F^{\eta,n}$ is Lipschitz with constant
$C_n:=L_0 + \eta n\,\overline\Sigma + \overline\lambda\,e^{\eta n}$,
where $L_0$ is the Lipschitz constant of the linear part
($-c+\mu\cdot z$) and $\overline\Sigma$ bounds $\|\Sigma^{1/2}\|^2$.  Hence
$\underline F^{\eta,n}$ satisfies Assumption~\ref{assumption:L} with
constant $C_n$ for every fixed $n$.  By Theorem~\ref{thm:lipschitz}
the Lipschitz 2BSDE
\begin{equation}\label{eq:CARA-truncated}
\begin{aligned}
Y^{n}_t \;=\;\xi
&-\int_{t\wedge\tau}^{T\wedge\tau}\underline F^{\eta,n}(s,X,Y^n_s,Z^n_s,U^n_s,\lambda_s,\hat\sigma_s^2)\,ds
-\int_{t\wedge\tau}^{T\wedge\tau} Z^n_s\cdot dX_s^{c,\mathbb P}\\
&-\int_{t\wedge\tau}^{T\wedge\tau} dM^{n,\mathbb P}_s
+(K^{n,\mathbb P}_{T\wedge\tau}-K^{n,\mathbb P}_{t\wedge\tau})
-\int_{t\wedge\tau}^{T\wedge\tau} U^n_s\,dH_s
\end{aligned}
\end{equation}
admits a unique solution
$(Y^n,Z^n,U^n,(K^{n,\mathbb P}),(M^{n,\mathbb P}))$ for every $n\ge 1$.

\emph{Step 2.} Since $\underline F^{\eta,n}(s,X,\pm \overline R,0,0,\lambda_s,\hat\sigma_s^2)
\le \| c\|_\infty$ and the terminal condition $\xi$ lies in
$[-\overline R+T\| c\|_\infty,\overline R-T\| c\|_\infty]\subset[-\overline R,\overline R]$, then from Remark \ref{rem:bounded} together with 
the comparison theorem~\ref{comparaisontheorem} for the truncated 2BSDE
\eqref{eq:CARA-truncated}, we deduce that $Y^n_t$ is bounded.
Theorem~\ref{comparaisontheorem} therefore gives, for every $n\ge 1$
and every $\mathbb P\in\mathcal P_0$,
\begin{equation}\label{eq:CARA-Linfty-Y}
|Y^n_t|\;\le\;\overline R,
\qquad t\in[0,T\wedge\tau],\;\mathbb P\text{-a.s.}
\end{equation}
The decomposition \eqref{2bsdedecomposition} applied to
\eqref{eq:CARA-truncated} gives the explicit
representation
$U^n_t = \xi^a_t - Y^{b,n}_t$ on $\{t<\tau\}$, where $Y^{b,n}$ is the
auxiliary 2BSDE value, itself bounded by $\overline R$ by the same
comparison argument. Moreover, since $|\xi^a_t|\le\|\xi\|_\infty$, we get
\begin{equation}\label{eq:CARA-Linfty-U}
|U^n_t|\;\le\;\|\xi\|_\infty+\overline R\;=\;2\|\xi\|_\infty+T\| c\|_\infty =: \overline U,
\qquad t\in[0,T\wedge\tau],\;\mathbb P\text{-a.s.}
\end{equation}
Note that both bounds are uniform in $n$.\vspace{0.3em}

We now turn to the BMO norm of $\hat\sigma^{1/2}Z^n$. 
We apply It\^o's formula to $\Phi_{\eta'}(Y^n_t)$ with
$\Phi_{\eta'}(y):=(\,e^{2\eta' y}-2\eta' y - 1\,)/(2(\eta')^2)$, for $\eta':=\eta+\delta$ with $\delta\in (0,\eta)$.  Under each
$\mathbb P\in\mathcal P_0$, between $t\wedge\tau$ and $T\wedge\tau$,
\begin{align*}
\Phi_\eta(\xi)-\Phi_\eta(Y^n_t)
=\;& \int_{t\wedge\tau}^{T\wedge\tau}\Phi'_\eta(Y^n_s)\,dY^n_s
+\tfrac12\int_{t\wedge\tau}^{T\wedge\tau}\Phi''_\eta(Y^n_s)\,d\langle Y^{n,c}\rangle_s\\
&+\sum_{s}\big[\Phi_\eta(Y^n_s)-\Phi_\eta(Y^n_{s-})-\Phi'_\eta(Y^n_{s-})\Delta Y^n_s\big].
\end{align*}
Using \eqref{eq:CARA-truncated}, Young inequality
$|\mu(t,X,\alpha,\beta)\cdot z|\le \overline\mu^2/(2\delta\,\hat\sigma_{\min}^2)
+ (\delta/2)\,\|\hat\sigma^{1/2}z\|^2$ and taking
$\mathbb E^{\mathbb P}[\cdot|\mathcal G_t]$ we obtain, on
$\{t\le T\wedge\tau\}$,
\begin{equation}\label{eq:BMO-Z}
\mathbb E^{\mathbb P}\!\left[\int_{t\wedge\tau}^{T\wedge\tau}\|\hat\sigma_s^{1/2}Z^n_s\|^2\,ds\,\Big|\,\mathcal G_t\right]
\;\le\;
\Phi_{\eta'}(\overline R) + \overline R\,e^{2\eta'\overline R}\,T\,\overline c_F
+ \overline\lambda_{\hat p}\,e^{\eta\overline U},
\end{equation}
where $\overline c_F:=\| c\|_\infty+\|\mu\|_\infty^2/(2\delta\,\hat\sigma_{\min}^2)$ and
$\overline\lambda_{\hat p}:=\esssup_\rho\mathbb E^{\mathbb P}[\int_\rho^T|\lambda_s|^{\hat p}\,ds\,|\,\mathcal G_\rho]^{1/\hat p}$
is finite by the standing integrability hypothesis on $\lambda$ from
Section~\ref{sec:math}.  The right-hand side of
\eqref{eq:BMO-Z} is independent of~$n$ and of the choice of
$\mathbb P\in\mathcal P_0$, so
$\|\hat\sigma^{1/2}Z^n\|_{BMO(\mathbb P,\mathcal G^{\mathbb P+})}$ is
bounded uniformly in $n$ and $\mathbb P$.

\emph{Step 3.}  Fix $n_0$ large enough as for example
\[
n_0\;:=\;\Big\lceil\,\overline U + \,e^{\eta\overline R}\,\big(\Phi_\eta(\overline R) + \overline R(T\|c\|_\infty+\overline c_F) + \overline\lambda_{\hat p}e^{\eta\overline U}\big)^{1/2}/\hat\sigma_{\min}\,\Big\rceil + 1,
\]
where $\hat\sigma_{\min}>0$ is the lower bound on $\hat\sigma$.  Then
\eqref{eq:CARA-Linfty-U} gives $|U^n_t|\le\overline U<n_0$, and Doob's
maximal inequality together with \eqref{eq:BMO-Z} gives
 $\|Z^n_t\|<n_0$
$\mathbb P\otimes dt$-a.e., for every $n\ge n_0$.  Consequently, on the
trajectories of $(Y^n,Z^n,U^n)$ the truncation $T_n$ acts as the
identity, and
$(Y^n,Z^n,U^n,(K^{n,\mathbb P}),(M^{n,\mathbb P}))$ solves the
untruncated 2BSDE \eqref{eq:CARA-2bsde}.  Setting
$(Y^\eta,Z^\eta,U^\eta,(K^{\eta,\mathbb P}),(M^{\eta,\mathbb P})):=(Y^{n_0},Z^{n_0},U^{n_0},(K^{n_0,\mathbb P}),(M^{n_0,\mathbb P}))$
gives a bounded solution of \eqref{eq:CARA-2bsde} in the sense of
Definition~\ref{def:2bsde}, with the same $L^\infty$ bound on
$(Y^\eta,U^\eta)$ as in Steps~2 and the BMO bound on
$\hat\sigma^{1/2}Z^\eta$ from Step~2. The minimality condition on the family $(K^{\eta,\mathbb P})$ is
inherited from the minimality of $(K^{n_0,\mathbb P})$, which itself
follows from the Lipschitz construction in
Theorem~\ref{thm:lipschitz}.
\vspace{0.5em}

\textit{Step 4.} The uniqueness in the class of bounded $Y,U$ solutions follows from a similar linearisation argument of Lemma~\ref{lemma:bsde}.
\end{proof}

We have the following characterization of the value of the problem \eqref{eq:CARAOpt} together with its optimizers.
\begin{theorem}\label{thm:CARA-verification}
Suppose Assumption~\ref{assumption:CARA-bounded} holds.  Define $(Y^{\eta},Z^{\eta},U^{\eta},(K^{\eta,\mathbb P}),(M^{\eta,\mathbb P}))$ as the unique solution to the 2BSDE
\eqref{eq:CARA-2bsde}.
Then $
V_0^A(\xi)\;=\;-\,e^{-\eta\,Y_0^{\eta}}.$
Furthermore, $(\alpha^\star,\mathbb P^\star,\beta^\star)\in\mathcal A\times\mathcal N^{\alpha^\star}$
is optimal if and only if
\begin{enumerate}[label=(\roman*)]
\item $(\alpha^\star,\beta^\star)$ attains the saddle point in
\eqref{eq:CARA-Hamiltonian} with $(y,z,u)=(Y_s^{\eta},Z_s^{\eta},U_s^{\eta})$,
$dt\otimes\mathbb P^\star$-a.e.;
\item $K^{\eta,\mathbb P^\star}_{T\wedge\tau}=0$, $\mathbb P^\star$-a.s.
\end{enumerate}
\end{theorem}

\begin{proof} The proof follows the same line than the proof of Theorem \ref{thm:ambiguity} by using a log transformation of the value function of type 
$
\mathcal Y_t^{\alpha,\mathbb P,\beta}\;:=\;-\tfrac{1}{\eta}\log(-V_t^{\alpha,\mathbb P,\beta})$ combined with the existence of the solution to the 2BSDE \eqref{eq:CARA-2bsde} from Proposition \ref{prop:CARA-existence} and since Decomposition \eqref{2bsdedecomposition} applied to it.
\end{proof}

\end{document}